\documentclass[twocolumn,superscriptaddress,prx,showkeys]{revtex4}

\usepackage{adjustbox}
\usepackage{amssymb}
\usepackage{bm}
\usepackage{color}
\usepackage{makecell}
\usepackage{microtype}
\usepackage{multirow}
\usepackage{paralist}
\usepackage{rotating}
\usepackage{indentfirst}
\usepackage{float}
\usepackage{fourier}

\usepackage{hyperref}
\usepackage[all]{xy}
\usepackage{xcolor}
\usepackage{dcolumn}
\usepackage{amsmath, amsthm, amssymb,amscd, mathrsfs, amsfonts, mathtools}
\usepackage{tikz}
\usepackage{diagbox}
\usetikzlibrary{matrix,arrows,decorations.pathmorphing}
\usetikzlibrary{backgrounds}
\usetikzlibrary{shapes.geometric}
\usepackage{circuitikz}
\usepackage{tcolorbox}
\usepackage{blkarray}

\usepackage{xspace}
\usepackage{ifthen}

\newboolean{showcomments}
\setboolean{showcomments}{true}
\ifthenelse{\boolean{showcomments}}
{
}

\newcommand{\wrt}{w.r.t.\@\xspace}

\def\rTo{\longrightarrow}

\def\<{\langle}
\def\>{\rangle}

\newcommand{\inner}[2]{\langle{#1}\rangle_{#2}} 

\let\cal\mathcal 
\let\bb\mathbb 

\def\im{\operatorname{im}}
\def\Hom{\operatorname{Hom}}

\newcommand{\tLap}[1]{{\cal L}^{\scalebox{.6}{(1)}}_{\scriptscriptstyle {#1}}} 
\newcommand{\bLap}[1]{{\cal L}^{\scalebox{.6}{(2)}}_{\scriptscriptstyle {#1}}} 
\newcommand{\Lap}[2]{{\cal L}^{\scalebox{.6}{({#1})}}_{\scriptscriptstyle {#2}}} 

\newcommand{\mixd}[1]{d^{\scalebox{.6}{({#1})}}} 
\newcommand{\mixb}[1]{\partial^{\scalebox{.6}{({#1})}}} 
\newcommand{\hbd}[1]{\partial^{\scalebox{.6}{(1)}}_{\scriptscriptstyle {#1}}} 
\newcommand{\vbd}[1]{\partial^{\scalebox{.6}{(2)}}_{\scriptscriptstyle {#1}}} 
\newcommand{\xbd}[2]{\partial^{\scalebox{.6}{({#1})}}_{\scriptscriptstyle {#2}}} 

\newcommand{\hcobd}[1]{\delta^{\scalebox{.6}{(1)}}_{\scriptscriptstyle {#1}}} 
\newcommand{\vcobd}[1]{\delta^{\scalebox{.6}{(2)}}_{\scriptscriptstyle {#1}}} 

\def\a{\alpha}
\def\b{\beta}

\def\g{\gamma}

\def\l{\lambda}

\def\s{\sigma}

\newcommand{\bsimplex}[4]{[#1_0^1, \ldots, #1_{#2}^1; #3_0^2, \ldots, #3_{#4}^2]}

\def\sgn{\operatorname{sgn}} 

\newcommand{\DiffBicomp}[3]{{{#1}}^{{{#2}\rightarrow {#3}}}} 

\newcommand{\outAdj}[2]{\smallfrown^{\scriptscriptstyle ({#1})}_{{#2}}} 
\newcommand{\inAdj}[2]{\smallsmile_{\scriptscriptstyle ({#1})}^{{#2}}} 

\newcommand{\hH}[1]{{\rm H}^{\scriptscriptstyle {(1)}}_{\scriptscriptstyle {#1}}} 
\newcommand{\vH}[1]{{\rm H}^{\scriptscriptstyle {(2)}}_{\scriptscriptstyle {#1}}} 
\newcommand{\xH}[2]{{\rm H}^{\scriptscriptstyle {({#1})}}_{\scriptscriptstyle {#2}}} 
\newcommand{\betti}[2]{\beta^{\scriptscriptstyle ({#1})}_{\scriptscriptstyle {#2}}}
\newcommand{\bettiVect}[1]{\beta_{\scriptscriptstyle {#1}}} 

\newcommand{\vcoh}[1]{{\rm H}_{\scalebox{.6}{(2)}}^{{#1}}}

\newcommand{\har}[2]{{}^{\scalebox{.6}{\rm ({#1})}}{\cal H}_{{#2}}} 

\tikzset{1simpl/.style={->,>=stealth,thick}}
\tikzset{vertex/.style = {circle, draw, fill=red, inner sep=1.5}}

{\endtcolorbox}

\newtheorem{thm}{Theorem}[section]

\newtheorem{lem}[thm]{Lemma}

\theoremstyle{definition}

\theoremstyle{remark}

\allowdisplaybreaks

\begin{document}

\title{Topology and spectral interconnectivities of higher-order multilayer networks}

\author{Elka\"ioum M.~Moutuou}
\email{elkaioum.moutuou@concordia.ca}
\affiliation{PERFORM Centre, Concordia University, Montreal, QC, H4B 1R6}
\affiliation{Gina Cody School of Engineering and Computer Science, Concordia University, Montreal, QC, H3G 1M8}
\affiliation{Department of Electrical and Computer Engineering, Concordia University, Montreal, QC, H3G 1M8}

\author{Oba\"i B.~K.~Ali}
\affiliation{PERFORM Centre, Concordia University, Montreal, QC, H4B 1R6}
\affiliation{Gina Cody School of Engineering and Computer Science, Concordia University, Montreal, QC, H3G 1M8}
\affiliation{Department of Electrical and Computer Engineering, Concordia University, Montreal, QC, H3G 1M8}
\affiliation{Department of Physics, Concordia University, Montreal, QC, H4B 1R6}


\author{Habib Benali}
\affiliation{PERFORM Centre, Concordia University, Montreal, QC, H4B 1R6}
\affiliation{Gina Cody School of Engineering and Computer Science, Concordia University, Montreal, QC, H3G 1M8}
\affiliation{Department of Electrical and Computer Engineering, Concordia University, Montreal, QC, H3G 1M8}

\begin{abstract}
Multilayer networks have permeated all the sciences as a powerful 
mathematical abstraction for interdependent heterogenous systems such as multimodal brain 
connectomes, transportation, ecological systems, and scientific collaboration.
 But describing such complex systems through a purely graph-theoretic formalism presupposes that the
 interactions that define the underlying infrastructures and support their functions 
are only pairwise-based; a strong assumption likely leading to oversimplifications. Indeed, 
most interdependent systems intrinsically involve
 higher-order intra- and inter-layer interactions. For instance, ecological systems involve
 interactions among groups within and in-between species, collaborations and citations 
 link teams of coauthors to articles and vice versa,  interactions might exist among groups 
of friends from different social networks, etc. While higher-order 
 interactions have been studied for monolayer systems through the language 
 of simplicial complexes and hypergraphs, a broad and systematic formalism incorporating them into the realm 
 of multilayer systems is still lacking.
Here, we introduce the concept of {\em crossimplicial multicomplexes} 
as a general formalism for modelling interdependent systems
 involving higher-order intra- and inter-layer connections. 
 Subsequently, we introduce {\em cross-homology} and its spectral 
 counterpart, the {\em cross-Laplacian} operators, to establish a 
 rigorous mathematical framework for quantifying global and local intra- and
  inter-layer topological structures in such systems. When applied to multilayer networks, these cross-Laplacians provide powerful methods for detecting
  clusters in one layer that are controlled by hubs in another layer. We call such hubs spectral cross-hubs (SCH) and define spectral persistence as a way to rank them according to their emergence along the cross-Laplacian spectra. 
   We illustrate our framework through 
  synthetic and empirical datasets.
 \end{abstract}

\maketitle

\section{Introduction}
{\em Multilayer networks}~\cite{boccaletti2014multilayer, de2013mathematical, kivela2014} have emerged over the last decade as a natural instrument in modelling myriads of heterogenous systems. They permeate all areas of science, as they provide a powerful abstraction of real-world phenomena made of interdependent sets of units interacting with each other through various channels. The concepts and computational methods they purvey have been the driving force to recent progress in the understanding of many highly sophisticated structures such as heterogeneous ecological systems~\cite{pilosof2017, timoteo2018}, spatiotemporal and multimodal human brain connectomes~\cite{griffa2017, pedersen2018, mandke2018}, gene-molecule-metabolite interactions~\cite{liu2020}, and interdisciplinary scientific collaborations~\cite{vasilyeva2021}. This success has led to a growing interdisciplinary research investigating fundamental properties and topological invariants in multilayer networks. 

  Some of the major challenges in the analysis of a multilayer network are to quantify the {\em importance} and {\em interdependence} among its different components and subsystems, and describe the topological structures of the underlying architecture to better grasp the dynamics and  information flows between its different network layers. Various approaches extending concepts, properties, and centrality indices from network science~\cite{newman2003structure, fortunato2016community}  have been developed, leading to tremendous results in many areas of science~\cite{boccaletti2014multilayer, flores2018, liu2020, sola2013eigenvector, sanchez2014dimensionality, timoteo2018, wu2019tensor, yuvaraj2021topological}. However, these approaches assume that inter- and intra-communications and relationships between the networks involved in such systems rely solely on node-based interactions. The resulting methods are therefore less insightful when the infrastructure is made up of higher-order {\em intra-} and {\em inter-connectivites} among node aggregations from different layers --- as it is the case for many phenomena. For example, heterogenous ecosystems are made up of interactions among groups of the same or different species, social networks often connect groups of people belonging to different circles, collaborations and citations form a higher-order multilayer network made of teams of co-authors interconnected to articles, etc. Many recent studies have explored higher-order interactions and structures in monolayer networks~\cite{giusti2015clique, benson2016higher, xu2016, benson2018simplicial, yin2018higher, iacopini2019simplicial, schaub2020, lucas2020multiorder, bianconi2021higher, ghorbanchian2021higher, young2021hypergraph, shi2021cliques, lotito2022higher, majhi2022dynamics} using different languages such as {\em simplicial complexes} and {\em hypergraphs}. But a general mathematical formalism for modelling and studying higher-order multilayer networks is still lacking.
  
\begin{figure}[!ht]
\begin{center}
\includegraphics[width=0.5\textwidth]{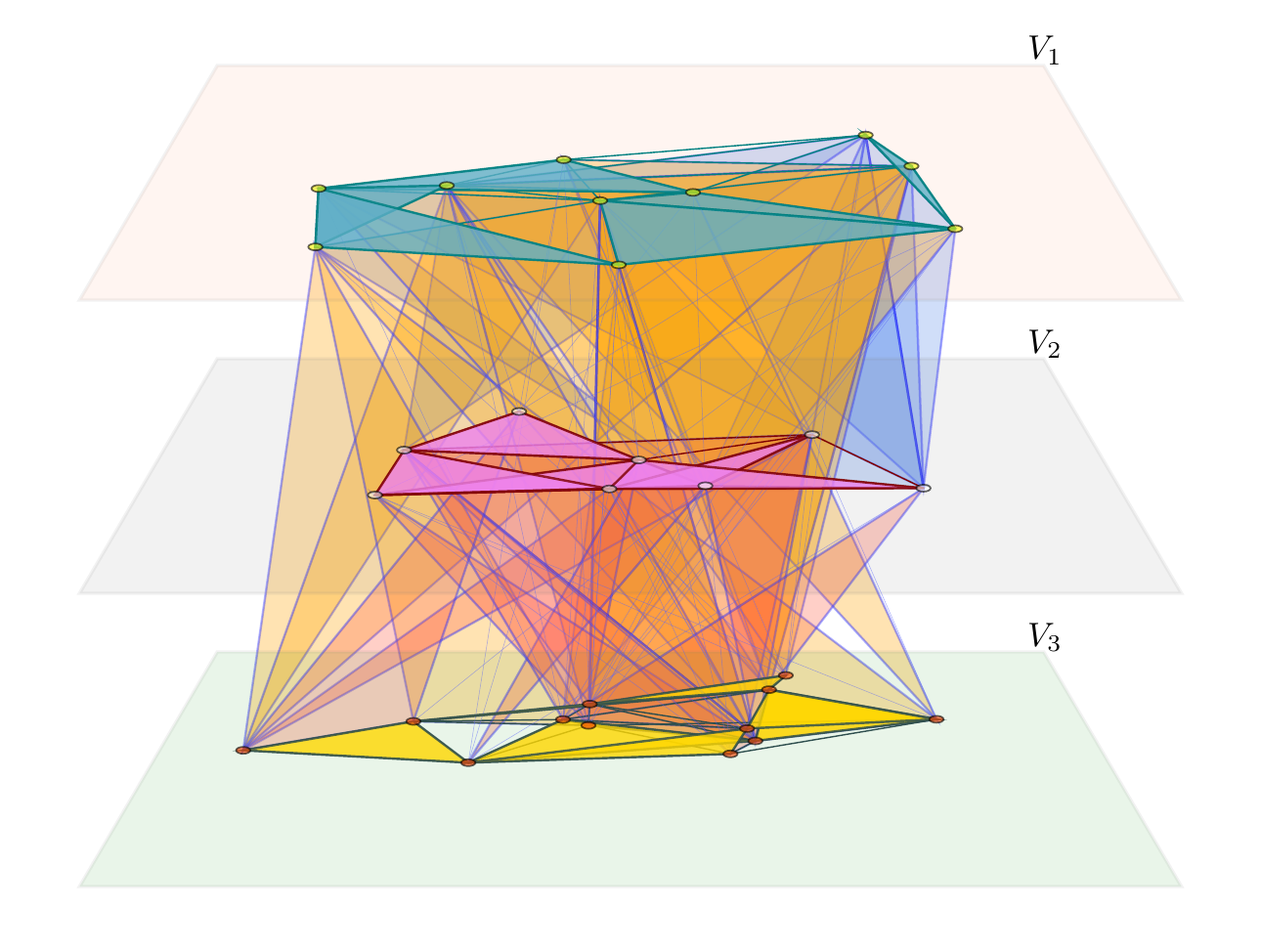} 
\end{center}
\caption{Shcematic of a 2--dimensional crossimplicial multicomplex $\cal X$ with 3 layers and 30 nodes in total; $\cal X$ consists of the vertex sets $V_1, V_2, V_3$ and the three CSBs $\cal X^{1,2}, \cal X^{1,3}, \cal X^{2,3}$ defined respectively on the products $V_1\times V_2, V_1\times V_3$, and $V_2\times V_3$. 
}
\label{fig:CSMC}
\end{figure}

  Our goal in this study is twofold. First, we propose a mathematical formalism that is rich enough to model and analyze multilayer complex systems involving higher-order connectivities within and in-between their subsystems. Second, we establish a unified framework for studying topological structures in such systems. This is done by introducing the concepts of {\em crossimplicial multicomplex}, {\em cross-homology}, {\em cross-Betti vectors}, and {\em cross-Laplacians}. Before we dive deeper into these notions, we shall give the intuition behind them by considering the simple case of an undirected two--layered network $\Gamma$; here $\Gamma$ consists of two graphs $(V_1,E_1), (V_2,E_2)$, where $V_1, V_2$ are the node sets of $\Gamma$, $E_s\subseteq V_s\times V_s, s=1,2$ are the sets of intra-layer edges, and a set $E_{1,2}\subseteq V_1\times V_2$ of {\em inter-layer edges}. Intuitively, $\Gamma$ might be seen as a system of interactions between two networks. And what that means is that the node set $V_1$ interacts not only with $V_2$ but also with the edge set $E_2$ and {\em vice versa}. Similarly, intra-layer edges in one layer interact with edges and triads in the other layer, and so on. This view suggests a more combinatorial representation by some kind of two-dimensional generalization of the fundamental notion of {\em simplicial complex} from Algebraic Topology~\cite{hatcher, lane1963homology}. The idea of {\em crossimplicial multicomplex} defined in the present work allows such a representation. In particular, when applied to a pairwise based multilayer network, this concept allows to incorporate, on the one hand, the {\em clique complexes}~\cite{lim2015hodge, schaub2020} corresponding to the network layers, and on the other, the clique complex representing the inter-layer relationships between the different layers into one single mathematical object. Morever, $\Gamma$ can be regarded through different lenses, and each view displays different kind of topological structures. The most naive perspective flattens the whole structure into a monolayer network without segregating the nodes and links from one layer or the other. Another viewpoint is of two networks with independent or interdependent topologies communicating with each other through the interlayer links. The rationale for defining cross-homology and the cross-Laplacians is to view $\Gamma$ as different systems each with its own intrinsic topology but in which nodes, links, etc., from one system have some restructuring power that allows them to impose and control additional topologies on the other. This means that   in a multilayer system, one layer network might display different topological structures depending on whether we look at it from its own point of view, from the lens of the other layers, or as a part of a whole aggregated structure. We describe this phenomenon by focusing on the spectra and eigenvectors of the lower degree cross-Laplacians. We shall however remark that our aim here is not to address a particular real-world problem but to provide broader mathematical settings that reveal and quantify the emergence of these structures in any type of multilayer network.
%

%
%
%
%
%
%

\section{Crossimplicial multicomplexes}
\paragraph*{\bf General definitions.} 
 Given two finite sets $V_1$ and $V_2$  and a pair of integers $k,l\ge -1$, a $(k,l)$--{\em crossimplex} $a$ in $V_1\times V_2$ is a subset $\{v_0^1,\ldots, v_k^1,v_0^2,\ldots, v_l^2\}$ of $V_1^{k+1}\times V_2^{l+1}$ where $v_i^s\in V_s$ for $s=1,2$. The points $v_i^1$ (resp. $v_j^2$) are the {\em vertices} of $a$ in $V_1$ (resp. $V_2$), and its {\em crossfaces} are its subsets of the form $\{v_0^1, \ldots, v_{i-1}^1, v_{i+1}^1, \ldots, v_k^1,v_0^2, \ldots, v_l^2\}$ for $0\le i \le k$ and $\{v_0^1,  \ldots, v_k^1,v_0^2, \ldots, v_{i-1}^2, v_{i+1}^2, \ldots, v_l^2\}$ for $0\le i \le l$. Note that here we have used the conventions that $V^n_1\times V_2^0 = V_1^n$ and $V_1^0\times V_2^n=V_2^n$.

\begin{figure*}[!ht]
\begin{center}
\begin{tikzpicture}[scale=0.6]

\begin{scope}[on background layer,opacity=0.5]
\fill[gray!20] (0,0) -- (4,0) -- (3,2) node[midway,above, black] {$ \ V_1$} -- (1,2) node[midway,above, black] {(a)}  -- cycle;
\fill[red!20] (0,-3) -- (4,-3) -- (3,-1)  node[midway,above, black] {$ \ V_2$} -- (1,-1) -- cycle;
\end{scope}

\node[vertex] (v001) at (2.5,1.8) {};
\node at (2.9,2) {$u_0^1$};
\node[vertex] (v011) at (2.5, 0.5) {};
\node at (2.5,0.9) {$u_1^1$};
\node[vertex] (v021) at (1,1) {};
\node at (1,1.4) {$u_2^1$};

\draw (v011) -- (v021);

\node[vertex, fill=yellow] (v002) at (3,-2.2) {};
\node at (3,-2.6) {$u_0^2$};
\node[vertex, fill=yellow] (v012) at (2,-1.5) {};
\node at (2,-1.1) {$u_1^2$};
\node[vertex, fill=yellow] (v022) at (1,-2) {};
\node at (1,-2.4) {$u_2^2$};

\draw (v002) -- (v012) -- (v022) -- (v002);
\begin{scope}[on background layer]
\fill[yellow!20, opacity=0.8] (v002.center) -- (v012.center) -- (v022.center) -- cycle;
\end{scope}

\begin{scope}[on background layer,opacity=0.5]
\fill[gray!20] (6,0) -- (10,0) -- (9,2) node[midway,above, black] {$ \ V_1$} -- (7,2) node[midway,above, black] {(b)}  -- cycle;
\fill[red!20] (6,-3) -- (10,-3) -- (9,-1)  node[midway,above, black] {$ \ V_2$} -- (7,-1) -- cycle;
\end{scope}
\node[vertex] (u01) at (8,1) {};
\node at (8, 1.4) {$v_0^1$};
\node[vertex, fill=yellow] (u02) at (8, -2) {};
\node at (8, -2.4) {$v_0^2$};

\draw[dashed] (u01) -- (u02);

\begin{scope}[on background layer,opacity=0.5]
\fill[gray!20] (12,0) -- (16,0) -- (15,2) node[midway,above, black] {$ \ V_1$} -- (13,2) node[midway,above, black] {(c)}  -- cycle;
\fill[red!20] (12,-3) -- (16,-3) -- (15,-1)  node[midway,above, black] {$ \ V_2$} -- (13,-1) -- cycle;
\end{scope}
\node[vertex] (v01) at (13.5,1) {};
\node at (13.5,1.4) {$w_1^1$};
\node[vertex] (v11) at (14.5,1) {};
\node at (14.5,1.4) {$w_0^1$};

\node[vertex, fill=yellow] (v02) at (14,-2) {};
\node at (14,-2.4) {$w_0^2$};

\draw (v01) -- (v11);
\draw[dashed] (v01) -- (v02) -- (v11);
\begin{scope}[on background layer]
\fill[blue!20, opacity=0.5] (v01.center) -- (v11.center) -- (v02.center) -- cycle;
\end{scope}

\begin{scope}[on background layer,opacity=0.5]
\fill[gray!20] (18,0) -- (22,0) -- (21,2) node[midway,above, black] {$ \ V_1$} -- (19,2) node[midway,above, black] {(d)}  -- cycle;
\fill[red!20] (18,-3) -- (22,-3) -- (21,-1)  node[midway,above, black] {$ \ V_2$} -- (19,-1) -- cycle;
\end{scope}

\node[vertex] (w01) at (19.5,1) {};
\node at (19.5, 1.4) {$w_1^1$};
\node[vertex] (w11) at (20.5, 1) {};
\node at (20.5, 1.4) {$w_0^1$};

\node[vertex, fill=yellow] (w02) at (19, -2) {};
\node at (19, -2.4) {$w_0^2$};
\node[vertex, fill=yellow] (w12) at (20.5,-1.5) {};
\node at (20.5, -1.9) {$w_1^2$};
\draw (w01) -- (w11);
\draw (w02) -- (w12);
\draw[dashed] (w11) -- (w12) -- (w01) -- (w02); 
\draw[dotted] (w02) -- (w11);
\begin{scope}[on background layer]
\fill[green!30, opacity=0.8] (w01.center) -- (w11.center) -- (w12.center) -- cycle;
\fill[green!20, opacity=0.8] (w02.center) -- (w12.center) -- (w01.center) -- cycle;
\end{scope}

\begin{scope}[on background layer,opacity=0.5]
\fill[gray!20] (24,0) -- (28,0) -- (27,2) node[midway,above, black] {$ \ V_1$} -- (25,2) node[midway,above, black] {(e)}  -- cycle;
\fill[red!20] (24,-3) -- (28,-3) -- (27,-1)  node[midway,above, black] {$ \ V_2$} -- (25,-1) -- cycle;
\end{scope}

\node[vertex] (z01) at (26, 1) {};
\node at (26, 1.6) {$z_0^1$};
\node[vertex, fill=yellow] (z02) at (25,-2.5) {};
\node at (24.6, -2.6) {$z_0^2$};
\node[vertex, fill=yellow] (z12) at (26, -1.5) {};
\node at (25.8, -1) {$z_1^2$};
\node[vertex, fill=yellow] (z22) at (27, -2) {}; 
\node at (27, -2.6) {$z_2^2$};

\draw[dashed] (z02) -- (z01) -- (z22);
\draw (z02) -- (z12) -- (z22) -- (z02);
\draw[dotted] (z01) -- (z12);
\begin{scope}[on background layer]
\fill[blue!40, opacity=0.8] (z01.center) -- (z02.center) -- (z22.center) -- cycle;
\end{scope}

\end{tikzpicture}
\end{center}
\caption{ {\bf Crossimplices.} Schematic of: (a) A $(0,-1)$--crossimplex (a top vertex), a $(1,-1)$--crossimplex (top horizontal edge), and a $(-1,2)$--crossimplex (bottom horizontal triangle); (b) a $(0,0)$--crossimplex (a cross-edge); (c) a $(1,0)$--crossimplex (a top cross-triangle); (d) a $(1,1)$--crossimplex (a cross-tetrahedron); and (e) a $(0,2)$--crossimplex (also a cross-tetrahedron). Notice that cross-edges are always oriented from the vertex of the top layer to the one in the bottom layer. Therefore, cross-edges belonging to a cross-triangle are always of opposite orientations with respect to any orientation of the cross-triangle. There are two types of cross-triangles: the $(1,0)$-crossimplices (top cross-triangles) and the $(0,1)$--crossimplices (the bottom cross-triangle). Moreover, there are three types of cross-tetrahedrons: the $(0,2)$--crossimplices, the $(2,0)$--crossimplices, and the $(1,1)$--crossimplices.}
\label{fig:crossimplices}
\end{figure*}
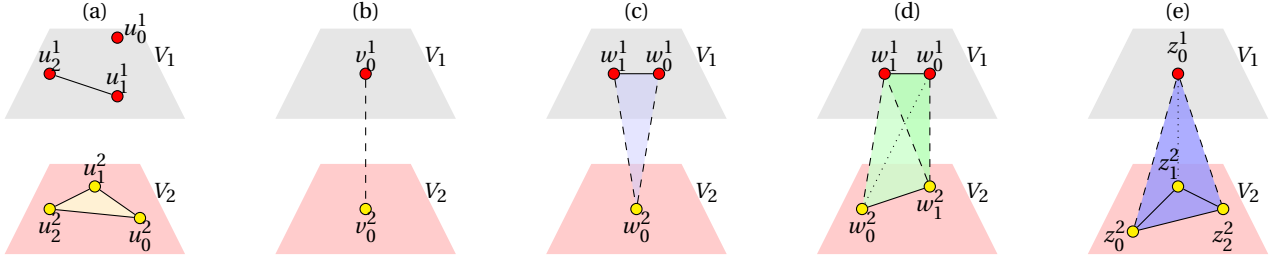

An {\em abstract crossimplicial bicomplex} $X$ (or a CSB) on $V_1$ and $V_2$ is a collection of crossimplices in $V_1\times V_2$ which is closed under the inclusion of crossfaces; {\em i.e.}, the crossface of a crossimplex is also a crossimplex. A crossimplex is {\em maximal} if it is not the crossface of any other crossimplex. $V_1$ and $V_2$ are called the vertex sets of $X$.

Given a CSB $X$, for fixed integers $k,l\ge 0$ we denote by $X_{k,l}$ the subset of all its $(k,l)$--crossimplices. We also use the notations
$X_{0,-1} = V_1, \ X_{-1,0} = V_2$, and $ X_{-1,-1} = \emptyset$. And recursively, $X_{k,-1}$ will denote the subset of crossimplices of the form $\{v_0^1, \ldots, v_k^1\}\subset V_1^{k+1}$ and $X_{-1,l}$ as the subset of crossimplices of the form $\{v_0^2, \ldots, v_l^2\}\subset V_2^{l+1}$. Such crossimplices will be referred to as {\em intralayer simplices} or {\em horizonal simplices}. We then obtain two simplicial complexes~\cite{hatcher} $X_{\bullet, -1}$ and $X_{-1,\bullet}$ that we will refer to as the {\em intralayer complexes}, and whose vertex sets are respectively $V_1$ and $V_2$. In particular, $X_{1,-1}$ and $X_{-1,1}$ are graphs with vertex sets $V_1$ and $V_2$, respectively.

The {\em dimension} of a $(k,l)$--crossimplex is $k+l+1$, and the dimension of the CSB $X$ is the dimension of its crossimplices of highest dimension. The $n$--{\em skeleton} of $X$ is the restriction of $X$ to the $(k,l)$--crossimplices such that $k+l+1\le n$. In particular, the $1$--skeleton of a CSB is a $2$--layered network, with $X_{0,0}$ being the set of interlayer links. Conversely, given a 2-layered network $\Gamma$ formed by two graphs $\Gamma_1=(V_1,E_1)$, $\Gamma_2= (E_2,V_2)$ with interlayer edge set $E_{1,2}\subset V_1\times V_2$, define a $(k,l)$--{\em clique} in $\Gamma$ as a pair $(\s_1, \s_2)$ where $\s_1$ is a $k$--clique in $\Gamma_1$ and $\s_2$ is an $l$--clique in $\Gamma_2$ with the property that $(i,j)\in E_{1,2}$ for every $i\in \s_1$ and $j\in \s_2$.  We define the {\em cross-clique bicomplex} $X$  associated to $\Gamma$ by letting $X_{k,l}$ to be the set of all $(k+1,l+1)$--{\em cliques} in $\Gamma$.

Now a {\em crossimplicial multicomplex} (CSM) $\cal X$  consists of a family of finite sets $V_s, s\in S\subseteq \bb N$, and a CSB $\cal X^{s,t}$ for each pair of distinct indices $s,t\in S$. It is {\em undirected} if the the sets of crossimplices in $\cal X^{s,t}$ and $\cal X^{t,s}$ are in one-to-one correspondence. In such a case, $\cal X$ is completely defined by the family of CSB $\cal X^{s,t}$ with $s<t$ (see Fig.~\ref{fig:CSMC} for a visualization of a 3--layer CSM). 

 

\paragraph*{\bf Orientation on crossimplices.}
An {\em orientation} of a $(k,l)$--crossimplex is an ordering choice over its vertices. When equipped with an orientation, the crossimplex is said to be {\em oriented} and will be represented as $[a]=\bsimplex{v}{k}{v}{l}$ if $k,l\ge 0$, or $[v_0^1, \ldots, v_k^1]$ (resp. $[v_0^2, \ldots, v_l^2]$) if  $k\ge 0$ and $l=-1$ (resp. $k=-1$ and $l\le 0$). We shall note that an orientation on crossimplices is just a choice purely made for computational purposes. Extending geometric representations from simplicial complexes, crossimplices can be represented as geometric objects.
Specifically, a $(0,-1)$--crossimplex is a vertex in the top layer; $(0,0)$--crossimplex  is a {\em cross-edge}  between layers $V_1$ and $V_2$; a $(1,-1)$--crossimplex (resp. $(-1,1)$--crossimplex) is a {\em horizontal edge} on $V_1$ (resp. $V_2$); a $(0,1)$--crossimplex  or a $(1,0)$--crossimplex  is a {\em cross-triangle}; a $(2,-1)$--crossimplex or $(-1,2)$--crossimplex is a {\em horizontal triangle} on layer $V_1$ or $V_2$; a $(3,-1)$--crossimplex or $(-1,3)$--crossimplex is a {\em horizontal tetrahedron} on $V_1$ or $V_2$; a $(1,1)$--crossimplex, a $(2,0)$--crossimplex, or a $(0,2)$--crossimplex,  is a {\em cross-tetrahedron}; and so on (see Fig.~\ref{fig:crossimplices} for illustrations). On the other hand, {\em horizontal} edges, triangles, tetrahedron, and so on, are just usual simplices on the horizontal complexes. One can think of a cross-edge as a connection between a vertex from one layer to a vertex on the other layer. In the same vein, a cross-triangle can be thought of as a connection between one vertex from one layer and two vertices on the other, and a cross-tetrahedron as a connection between either two vertices from one layer and two vertices on the other, or one vertex from one layer to three vertices on the other.
\paragraph*{\bf Weighted CSBs.}
A {\em weight} on a CSB $X$ is a positive function
$w: \bigcup_{k,l}X_{k,l} \rTo \bb R^+$
that does not depend on the orientations of crossimplices. A {\em weighted CSB} is one that is endowed with a weight function. The {\em weight of a crossimplex} $a\in X$ is the number $w(a)$.

%

\section{Topological descriptors}
\paragraph*{\bf Cross-boundaries.}
A CSB $X$ defines a {\em bisimplicial set}~\cite{goerss2009simplicial, moerdijk1989} by considering respectively the {\em top} and {\em bottom crossface maps} $\mixd{1}_{i|k,l}: X_{k,l} \rTo X_{k-1, l}$ and $\mixd{2}_{i|k,l}: X_{k,l} \rTo X_{k,l-1}$
by
\begin{equation}
\resizebox{.9\hsize}{!}{$
  \begin{aligned}
  \mixd{1}_{i|k,l}(\bsimplex{v}{k}{v}{l}) & = [v_0^1, \ldots, \widehat{v_i^1},\ldots, v_k^1; v_0^2, \ldots, v_l^2] \\
  \mixd{2}_{i|k,l}(\bsimplex{v}{k}{v}{l}) & =  [v_0^1, \ldots, v_k^1; v_0^2, \ldots, \widehat{v_i^2}, \ldots, v_k^2]
  \end{aligned}
  $}
\end{equation}
where the hat over a vertex means dropping the vertex. Moreover, for a fixed $l\ge -1$, $X_{\bullet, l}=(X_{k,l})_{k\ge -1}$ is a simplicial complex. Similarly, $X_{k,\bullet}=(X_{k,l})_{l\ge -1}$ is a simplicial complex. Observe that if $a=\{v_0^1, \ldots, v_k^1, v_0^2, \ldots, v_l^2\} \in X_{k,l}$, then $a^{(1)} = \{v_0^1, \ldots, v_k^1\}\in X_{k,-1}$ and $a^{(2)}=\{v_0^2, \ldots, v_l^2\}\in X_{-1,l}$. We will refer to $a^{(1)}$ and $a^{(2)}$ as {\em the top horizontal face} and {\em the bottom horizontal face} of $a$, respectively.  Conversely, two horizontal simplices $v^1\in X_{k,-1}$ and $v^2\in X_{-1,l}$ are said to be {\em interconnected} in $X$ if they are respectively the top and bottom horizontal faces of a $(k,l)$--crossimplex $a$. We then write $v^1\sim v^2$. This is basically equivalent to requiring that if $v^1=\{v_0^1, \ldots, v_k^1\}$ and $v^2=\{v_0^2,\ldots, v_l^2\}$ then $\{v_0^1, \ldots, v_k^1,v_0^2,\ldots, v_l^2\}\in X_{k,l}$. If $a=\{v_0^1, \ldots, v_k^1, v_0^2,\ldots, v_l^2\}\in X_{k,l}$, we define its {\em top cross-boundary} $\hbd{}a$ as the subset of $X_{k-1,l}$ consisting of all the top crossfaces of $a$; {\em i.e.}, all the $(k-1,l)$--crossimplices of the form $\mixd{1}_{i|k,l}[a]$ for $i=0,\ldots, k$. Analogously, its {\em bottom cross-boundary} $\vbd{}a\subseteq X_{k,l-1}$ is the subset of all its bottom crossfaces $\mixd{2}_{i|k,l}[a], i=0, \ldots, l$. 

Now two $(k,l)$--crossimplices $a,b\in X_{k,l}$ are said to be:

\begin{itemize}
\item {\em top-outer (TO) adjacent}, and we write $a\outAdj{1}{}b$ or $a\outAdj{1}{c}b$, if both are top crossfaces of a $(k+1,l)$--crossimplex $c$; in other words $a,b\in \hbd{}c$;

\item {\em top-inner (TI) adjacent}, and we write $a\inAdj{1}{}b$ or $a\inAdj{1}{d}b$, if there exists a $(k-1,l)$--crossimplex $d\in X_{k-1,l}$ which is a top crossface of both $a$ and $b$; {\em i.e.}, $d\in \hbd{}a\cap \hbd{}b$;  

\item {\em bottom-outer (BO) adjacent}, and we write $a\outAdj{2}{}b$ or $a\outAdj{2}{c}b$, if both are bottom crossfaces of a $(k,l+1)$--crossimplex $c\in X_{k,l+1}$; that is to say $a,b\in \vbd{}c$; and 

 \item {\em bottom-inner (BI) adjacent}, and we write $a\inAdj{2}{}b$ or $a\inAdj{2}{d}b$, if there exists a $(k,l-1)$--crossimplex $f\in X_{k,l-1}$ which is a bottom face of both $a$ and $b$; that is $d\in \vbd{}a\cap \vbd{}b$.
\end{itemize}
\paragraph*{\bf Degrees of crossimplices.}

Given a weight function $w$ on $X$, we define the following degrees of a $(k,l)$--crossimplex $a$ relative to $w$.
\begin{itemize}
  \item The {\em TO degree} of $a$ is the number 
  \begin{eqnarray}\label{eq:TO-degree}
    \deg_{TO}(a) = \deg_{TO}(a,w):= \sum_{\overset{a'\in X_{k+1,l}}{a\in \hbd{}a'}}w(a').
  \end{eqnarray}
  \item Similarly, the {\em TI degree} of $a$ is defined as
  \begin{eqnarray}\label{eq:TI-degree}
    \deg_{TI}(a) = \deg_{TI}(a,w) := \sum_{\overset{c\in X_{k-1,l}}{c\in \hbd{}a}}\frac{1}{w(c)}.
  \end{eqnarray}
  
  \item Analogously, the {\em BO degree} of $a$ is given by
  \begin{eqnarray}\label{eq:BO-degree}
    \deg_{BO}(a) = \deg_{BO}(a,w) := \sum_{\overset{a'\in X_{k,l+1}}{a\in \vbd{}a'}}w(a').
  \end{eqnarray}
  \item And the {\em BO degree} of $a$ is
  \begin{eqnarray}\label{eq:BI-degree}
    \deg_{BI}(a) = \deg_{BI}(a,w) := \sum_{\overset{c\in X_{k,l-1}}{c\in \vbd{}a}}\frac{1}{w(c)}.
  \end{eqnarray}

\end{itemize}
Observe that in the particular case where the weight function is everywhere equal to one, the TO degree of $a$ is precisely the number of $(k+1,l)$--crossimplices in $X$ of which $a$ is top cross-face, while $\deg_{TI}(a)$ is the number of top cross-faces of $a$, which equals to $k+1$. Analogous observation can be made about the BO and BI degrees.

\paragraph*{\bf Cross-homology groups.}

Define the space $C_{k,l}$ of $(k,l)$--{\em cross-chains} as the real vector space generated by all oriented $(k,l)$--crossimplexes in $\cal X$. The {\em top} and {\em bottom cross-boundary operators} $\hbd{k,l}: C_{k,l}\rTo C_{k-1,l}$, $\vbd{k,l}:C_{k,l}\rTo C_{k,l-1}$
are then defined as follows by the formula 
\begin{equation}\label{eq:boundaries}
  \xbd{s}{k,l}([a]) := \sum_{b\in \xbd{s}{}a}\sgn(b,\xbd{s}{}a)[b],   \\
\end{equation}
for $s=1,2$ and a generator $a\in X_{k,l}$, where $\sgn(b,\xbd{s}{}a)$ is the sign of the orientation of $b$ in $\xbd{s}{}a$; that is, if $b=\mixd{1}_{i|k,l}[a]$, then $\sgn(b,\hbd{}a):=(-1)^i$, and we define $\sgn(b,\vbd{}a)$ in a similar fashion.

It is straightforward to see that in particular 
\[
\hbd{k,-1}: C_{k,-1}\rTo C_{k-1,-1}, k\ge 0,
\]
and
\[
\vbd{-1,l}: C_{-1,l}\rTo C_{-1, l}, l\ge 0
\]
are the usual boundary maps of simplicial complexes. For this reason, we will put more focus on the mixed case where both $l$ and $k$ are non-negative. We will often drop the indices and just write $\mixb{1}$ and $\mixb{2}$ to avoid cumbersome notations. To see how these maps operates, let us compute for instance the images of the crossimplices  (b), (c), (d) and (e) illustrated in Fig.~\ref{fig:crossimplices}. We get:

\[
\left\{
\begin{array}{ll}
  \hbd{0,0}[v_0^1;v_0^2] & = [v_0^2]\in C_{-1,0} \\
   & \\
  \vbd{0,0}[v_0^1;v_0^2] & = - [v_0^1]\in C_{0,-1};
\end{array} 
\right.
\]
\\
\[
\left\{
\begin{array}{ll}
\hbd{1,0}[w_0^1,w_1^1;w_0^2] & = [w_1^1;w_0^2] - [w_0^1;w_0^2] \in C_{0,0}, \\
 & \\
\vbd{1,0}[w_0^1,w_1^1;w_0^2] & = [w_0^1,w_1^1] \in C_{1,-1};
\end{array}
\right. 
 \]
 \\
 \[
\left\{
\begin{array}{ll}
  \hbd{1,1}[w_0^1,w_1^1;w_0^2,w_1^2] & = [w_1^1;w_0^2, w_1^2] - [w_0^1; w_0^2,w_0^2] \in C_{0,1},\\
   & \\
   \vbd{1,1}[w_0^1,w_1^1;w_0^2,w_1^2] & = [w_0^1, w_1^1;w_1^2] - [w_0^1, w_1^1;w_0^2] \in C_{1,0}; 
\end{array}
\right. 
\]
\\
\[
\left\{
\begin{array}{ll}
\hbd{0,2}[z_0^1;z_0^2,z_1^2,z_2^2]  = & [z_0^2,z_1^2,z_2^2] \in C_{-1,2} \\ 
& \\
\vbd{0,2}[z_0^1;z_0^2,z_1^2,z_2^2]  = & [z_0^1;z_1^2,z_2^2] - [z_0^1;z_0^2,z_2^2]  \\
& + [z_0^1;z_0^2,z_1^2] \in C_{0,1}.
\end{array}
\right. 
\]

Notice that $\hbd{0,-1}=\vbd{-1,0}=0$. Moreover, by simple calculations from~\eqref{eq:boundaries}, it is easy to check that $\hbd{k-1,l}\hbd{k,l} = 0$ and  $\vbd{k,l-1}\vbd{k,l}  = 0$,
which allows to define the {\em top} and {\em bottom} $(k,l)$--{\em cross-homology groups} of $X$ as the quotients
\[
\hH{k,l}(X) := \ker \hbd{k,l}/\im \hbd{k+1,l}, \ {\rm and}
 \]
 
 \[
 \vH{k,l}(X):= \ker \vbd{k,l}/\im\vbd{k,l+1}.
 \]
 For for $k\ge 0$ and $l\le 0$,  $\hbd{k,-1}$  and $\vbd{-1,l}$ 
are the usual boundary maps of simplicial complexes~\cite{hatcher}. Therefore $\hH{k,-1}(X)$ and $\vH{-1,l}(X)$ are the usual homology groups~\cite{hatcher, lane1963homology} of the simplicial complexes $X_{\bullet, -1}$ and $X_{-1,\bullet}$, respectively.

\paragraph*{\bf Cross-Betti vectors.}

The cross-homology groups are completely determined by their dimensions, the {\em top} and {\em bottom} $(k,l)$--{\em cross-Betti numbers} $\betti{s}{k,l}(X) = \dim \xH{s}{k,l}(X)$, $s=1,2$.
In particular, $\betti{1}{k,-1}$ and $\betti{2}{-1,l}$ are the usual Betti numbers for the horizontal simplicial complexes~\cite{hatcher}. The couple $\b_{k,l}=(\betti{1}{k,l}, \betti{2}{k,l})$ is the $(k,l)$--{\em cross-Betti vector} of $X$ and can be computed using basic Linear Algebra. 
These vectors are descriptors of the topologies of both the horizontal complexes and their inter-connections. For instance, $\b_{0,-1}$ and $\b_{-1,0}$ encode the connectivities within and in-between the $1$--skeletons of the horizontal complexes associated to $X$. Precisely, $\betti{1}{0,-1}$ is the number of connected components of the graph $X_{1, -1}$ and $\betti{2}{0,-1}$ is the number of nodes in $V_1$ with no interconnections with any nodes in $V_2$. Similarly, $\betti{1}{-1,0}$ is the number of nodes in $V_2$ with no interconnections with any nodes in $V_1$, while $\betti{2}{-1,0}$ is the number of connected components of the bottom horizontal graph $X_{-1,1}$. 
Furthermore, $\b_{1,-1}$ counts simultaneously the number of loops in $X_{1,-1}$ and the number of its intralayer links that do not belong to cross-triangles formed with the graph $X_{-1,1}$. Analogous topological information is provided by $\b_{-1,1}$. Also, $\b_{0,0}$ measures the extent to which individual nodes of one complex layer serve as communication channels between different hubs from the other layer. More precisely, an element in $\hH{0,0}(X)$ represents either an interlayer $1$--dimensional loop formed by a path in $X_{1,-1}$ whose end-nodes interconnect with the same node in $V_2$, or two connected components in the top complex communicating with each other through a node in the bottom complex. In fact, $\b_{0,0}$ counts the shortest paths of length $2$ between nodes within one layer passing through a node from the other layer and not belonging to the cross-boundaries of cross-triangles; we call such paths {\em cones}. Put differently, $\b_{0,0}$ quantifies node clusters in one layer that are "controlled" by nodes in the other layer. Detailed proof of this description is provided in Appendix~\ref{ap:cones}.


\begin{table}[!ht]
\centering
\begin{tabular}{c| c c c}
                  & $\cal X^{1,2}$ & $\cal X^{1,3}$ & $\cal X^{2,3}$ \\
\hline
$\b^{\otimes}_{0,-1} $ & $(1,0)$ & $(1,0)$ & $(1,0)$  \\
$\b^{\otimes}_{1,-1} $ & $(13,21)$ & $(13,21)$ & $(6,14)$ \\
$\b^{\otimes}_{-1,0} $ & $(0,1)$ & $(1,1)$ & $(0,1)$ \\
$\b^{\otimes}_{-1,1} $ & $(17,6)$ & $(29,16)$ & $(29,16)$ \\
$\b^{\otimes}_{0,0} $ & $(11,14)$ & $(20,24)$ & $(19,23)$ 

\end{tabular}
\caption{{\bf Cross-Betti table}. The cross-Betti table for the CSM of Figure~\ref{fig:CSMC}. The table quantifies the connectedness of the three horizontal complexes, the number of cycles in each of them, the number of nodes in each layer that are not connected to the other layers, the number of intra-layer edges not belonging to any cross-triangles, as well as the number of paths of length 2 connecting nodes in one layer and passing through a node from another layer.}
\label{tab:betti-table}

\end{table}

Now, given a CSM $\cal X$, its {\em cross-Betti table} $\b^{\otimes}_{k,l}$ is obtained by computing all the cross-Betti vectors of all its underlying CSB's. Computation of the cross-Betti table of the CSM of Fig.~\ref{fig:CSMC} is presented in Table~\ref{tab:betti-table}. 


\begin{figure}[!ht]
\centering
\begin{tikzpicture}[scale=.8]

\begin{scope}[on background layer,opacity=0.5]
\fill[gray!20] (0,0) -- (10,0) -- (8,2) node[midway,above, black] {$ \ V_1$} -- (1,2)  -- cycle;
\fill[red!20] (0,-3) -- (10,-3) -- (8,-1)  node[midway,above, black] {$ \ V_2$} -- (1,-1) -- cycle;
\end{scope}

\node[vertex] (u0) at (8.5,0.5) {};
\node at (8.9,0.5) {$v^1_0$};
\node[vertex] (u1) at (7.5,1) {};
\node at (7.5,1.4) {$v^1_1$};
\node[vertex] (u2) at (6.5,0.2) {};
\node at (6.5,0.6) {$v^1_2$};
\node[vertex] (u3) at (5,1.4) {};
\node at (5,1.8) {$v^1_3$};
\node[vertex] (u4) at (4,0.5) {};
\node at (4,0.9) {$v^1_4$};
\node[vertex] (u5) at (3,1.5) {};
\node at (3.4,1.5) {$v^1_5$};
\node[vertex] (u6) at (2, 0.5) {};
\node at (1.6, 0.5) {$v^1_6$};
\node[vertex] (u7) at (1,1) {};
\node at (0.7, 1) {$v^1_7$};

\node[vertex, fill=yellow] (v0) at (8,-2.5) {};
\node at (8.4, -2.5) {$v^2_0$};
\node[vertex, fill=yellow] (v1) at (6,-1.5) {};
\node at (6.4, -1.5) {$v^2_1$};
\node[vertex, fill=yellow] (v2) at (5,-2.5) {};
\node at (5,-2.8) {$v^2_2$};
\node[vertex, fill=yellow] (v3) at (3,-2) {};
\node at (3, -2.4) {$v^2_3$};
\node[vertex, fill=yellow] (v4) at (2, -1.5) {};
\node at (1.6, -1.5) {$v^2_4$};
\node[vertex, fill=yellow] (v5) at (1, -2.5) {};
\node at (0.6, -2.5) {$v^2_5$};

\draw (u0) -- (u1) -- (u2) -- (u3) -- (u4);
\draw (u1) -- (u3);
\draw (u5) -- (u6) -- (u7) -- (u5);
\draw (v2) -- (v1) -- (v0) -- (v2) -- (v3) -- (v4);
\draw (v4) -- (v5) -- (v3);

\draw[dashed] (u0) -- (v1) -- (u1);
\draw[dashed] (u2) -- (v1) -- (u4);
\draw[dashed] (u4) -- (v4) -- (u6) ;
\draw[dashed] (v1) -- (u6) -- (v2);
\draw[dashed] (u6) -- (v5);

\begin{scope}[on background layer]
\fill[magenta!70] (u5.center) -- (u6.center) -- (u7.center) -- cycle;
\fill[blue!20, opacity=0.6] (u0.center) -- (v1.center) -- (u1.center) -- cycle;
\fill[blue!20, opacity=0.5] (u1.center) -- (v1.center) -- (u2.center) -- cycle;
\fill[green!50, opacity=0.5] (u6.center) -- (v1.center) -- (v2.center) -- cycle;
\fill[green!60, opacity=0.5] (u6.center) -- (v4.center) -- (v5.center) -- cycle;
\fill[yellow!20, opacity=0.8] (v0.center) -- (v1.center) -- (v2.center) -- cycle;
\fill[yellow!20, opacity=0.8] (v4.center) -- (v5.center) -- (v3.center) -- cycle;

\end{scope}

\end{tikzpicture}
\caption{{\bf Cross-Betti vectors.}  Schematic of a $2$--dimensional CSB with 14 nodes in total, and whose oriented maximal crossimplices are the intralayer triangle $[v^1_5, v^1_6, v^1_7] $ in $X_{2,-1}$,  the intralayer edges $[v^1_0,v^1_1], [v^1_1,v^1_2], [v^1_1,v^1_3], [v^1_2,v^1_3]$ in $X_{1,-1}$, the bottom intralayer triangles $[v^2_0,v^2_1,v^2_2], [v^2_3,v^2_4,v^2_5]$ in $X_{-1,2}$, and the intralayer edge $[v^2_2,v^2_3]$ in $X_{-1,1}$, the cross-triangles $[v^1_0,v^1_1;v^2_1], [v^1_1,v^1_2;v^2_1]$ in $X_{1,0}$, $[v^1_6;v^2_1,v^2_2], [v^1_6;v^2_4,v^2_5]$ in $X_{0,1}$, and the cross-edges $[v^1_4;v^2_1], [v^1_4;v^2_4]$ in  $X_{0,0}$.}
\label{fig:cross-betti}
\end{figure}
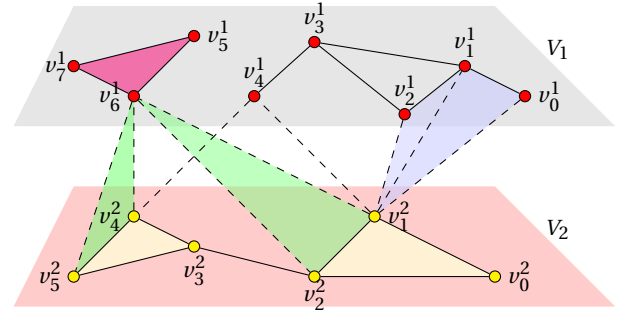

To illustrate what the cross-Betti vectors represent, we consider the simple $2$--dimensional CSB $X$ of Fig.~\ref{fig:cross-betti}. We get $\betti{1}{0,-1}=2, \betti{1}{1,-1}=1$, and $\betti{2}{-1,0}=1, \betti{2}{-1,1}=0$; which reflects the fact the top layer has $2$ connected components and $1$ cycle, while the bottom one has one component and no cycles. Moreover, $3$ top nodes are not interconnected to the bottom complex, $6$ top edges are not top faces of cross-triangles, $2$ bottom nodes are not interconnect to the top layer, and $5$ bottom edges are not bottom faces of cross-triangles. This information is encoded in $\b_{0,-1}=(2,3), \b_{1,-1}=(1,6), \b_{-1,0}=(2,1)$ and $\b_{-1,1}=(5,0)$. There are $3$ generating interlayer cycles, two of which are formed by an intralayer path in the bottom layer and a node in the top layer ($v^1_4$ and $v^1_6$), and the other one is formed by an intralayer path in the top layer and a node ($v^2_1$) in the bottom layer. Moreover, the two nodes $v^2_1$ and $v^2_4$ of $V_2$ interconnect the two separated components of the top layer; they serve as {\em cross-hubs}: removing both nodes eliminates all communications between the two components of the top layer. Cross-hubs and these types of interlayer cycles are exactly what $\b_{0,0}$ encodes. Specifically, by computing the cross-homology of $X$ we get $\betti{1}{0,0}=3$ which count the cycle $v^1_2 - v^1_3 - v^1_4 - v^2_1 - v^1_2$ and the nodes $v^2_4$ and $v^2_1$ that interconnect $v^1_4$ to $v^1_6$ and $v^1_2$ to $v^1_6$, $\betti{2}{0,0}=2$ counting the interlayer cycles $v^1_4 - v^2_1 - v^2_2 - v^2_3 - v^2_4 - v^1_4$ and $v^1_6 - v^2_2 - v^2_3 - v^2_4 - v^1_6$.  In each of these cycles, the top node allows a shortest (interlayer) path between the end-points of the involved intralayer path. 

Using algebraic-topological methods to calculate the cross-Betti vectors for larger multicomplexes can quickly become computationally heavy. We provide powerful linear-algebraic tools that not only allow to compute easily the $\b_{k,l}$'s, but also tell exactly where the topological structures being counted are located within the multicomplex.


\section{Spectral descriptors}
\paragraph*{\bf Cross-forms.}
 Denote by $C^{k,l}:=C^{k,l}(X,\bb R)$ the dual space $\Hom_{\bb R}(C_{k,l},\bb R)$ of the real vector space $C_{k,l}$. Namely, $C^{k,l}$ is the vector space of real linear functionals $\phi: C_{k,l}\rTo \bb R$. 
We will refer to such functionals as $(k,l)$--{\em forms} or {\em cross-forms}  on $X$. In particular, $(k,-1)$--forms correspond to $k$--forms on the simplicial complex $X_{\bullet, -1}$, and $(-1,l)$--forms are $l$--forms on the complex $X_{-1,\bullet}$. We have $C^{-1,-1}=0$, and by convention we set $C^{k,l}(\emptyset)=0$. 

Notice that a natural basis of $C^{k,l}$ is given by the set of linear forms $$\{e_a:C_{k,l}\rTo \bb R, \  a\in X_{k,l}\},$$ called {\em elementary cross-forms}, where 
\[
e_a(b) = \left\{\begin{array}{ll}1, & {\rm if \ } a=b, \\ 0, & {\rm otherwise}, \end{array} \right.
\]
which naturally identify $C^{k,l}$ with $C_{k,l}$. Now, define the maps $\hcobd{k,l}:C^{k,l}\rTo C^{k+1,l}$ and $\vcobd{k,l}:C^{k,l}\rTo C^{k,l+1}$
%
by the following equations
\begin{eqnarray}
\begin{aligned}
  \hcobd{k,l}\phi([a]) & = \sum_{b\in \hbd{}a}\sgn(b,\hbd{}a)\phi([b]), \\
   \vcobd{k,l}\phi([c]) & = \sum_{d\in \vbd{}c}\sgn(d,\vbd{}c)\phi([d]),
  \end{aligned}
\end{eqnarray}
for $\phi\in C^{k,l}, a\in X_{k+1,l}$ and $c\in X_{k,l+1}$. Next, given a weight $w$ on $X$, we get an inner-product on cross-forms by setting
\begin{eqnarray}\label{eq:weight-inner-product}
  \inner{\phi,\psi}{k,l} := \sum_{a\in X_{k,l}}w(a)\phi(a)\psi(a), \ {\rm for \ } \phi, \psi \in C^{k,l}.
\end{eqnarray}
It can been seen that, with respect to this inner-product, elementary cross-forms form an orthogonal basis, and by simple calculations,
 the dual maps are given by

\begin{eqnarray}\label{eq:dual-coboundary-formula}
  (\hcobd{k,l})^*\phi([a]) = \sum_{\overset{a'\in X_{k+1,l}}{a\in \hbd{}a'}}\frac{w(a')}{w(a)}\sgn(a,\hbd{}a')\phi([a']),
\end{eqnarray}
for $\phi\in C^{k+1,l}, a\in X_{k,l}$. And obviously we get a similar formula for the dual $(\vcobd{k,l})^*$.

\paragraph*{\bf The cross-Laplacian operators.}

%
 Identifying $C_{k,l}$ with $C^{k,l}$ and equipping it with an inner product as~\eqref{eq:weight-inner-product}, we define the following self-adjoint linear operators on $C_{k,l}$ for all $k,l\ge -1$:

    - the {\em top $(k,l)$--cross-Laplacian} 
   \[
   \Lap{T}{k,l} := (\hcobd{k,l})^*\hcobd{k,l} + \hcobd{k-1,l}(\hcobd{k-1,l})^*;
   \]
   
  - and the {\em bottom $(k,l)$--cross-Laplacian} 
   \[
   \Lap{B}{k,l} := (\vcobd{k,l})^*\vcobd{k,l} + \vcobd{k,l-1}(\vcobd{k,l-1})^*.
   \]

Being defined on finite dimensional spaces, these operators can be represented as square matrices indexed over crossimplices. Specifically, denoting $N_{k,l} = |X_{k,l}|$, $\Lap{T}{k,l}$ is represented by positive definite $N_{k,l}\times N_{k,l}$--matrices (see Appendix~\ref{ap:matrices}) for the general expressions). 
%
%

Moreover, the null-spaces, the elements of which we call {\em harrmonic} {\em cross-forms}, are easily seen to be in one-to-one correspondence with cross-cycles on $X$. Namely, we have the following isomorphisms (see Appendix~\ref{ap:harmonic} for the proof) 
\[
\hH{k,l}(X) \cong \ker\Lap{T}{k,l}, \ \vH{k,l}(X)\cong \ker\Lap{B}{k,l}.
\]
It follows that in order to compute the cross-Betti vectors, it suffices to determine the dimensions of the eigenspaces of the zero-eigenvalues of the cross-Laplacians. 

It should be noted that in addition to being much easier to implement, the spectral method to compute cross-homology has the advantage of providing a geometric representation of the cross-Betti numbers through eigenvectors. But before we see how this works, let's make a few observations. Notice that $\Lap{T}{0,-1}$ and $\Lap{B}{-1,0}$ are the usual graph Laplacians of degree $0$ for the horizontal complexes. And more generally, $\Lap{T}{k,-1}$  and $\Lap{B}{-1,l}$ are the combinatorial higher Hodge Laplacians~\cite{horak2013spectra, lim2015hodge, schaub2020} of degree $k$ and $l$, respectively, for the horizontal simplicial complexes. Furthermore, $\Lap{B}{k,-1}$ (resp. $\Lap{T}{-1,l}$) detects the $k$--simplices (resp. $l$--simplices) in the top (resp. bottom) layer complex that are not top (resp. bottom) faces of $(k,0)$--crossimplices (resp. $(0,l)$--crossimplices). Moreover, one can see that $\Lap{B}{k,-1}$ is the diagonal matrix indexed over the $k$--simplices on the top complex and whose diagonal entries are the BO degrees. Similarly, $\Lap{T}{-1,l}$ is the diagonal matrix whose diagonal entries are the TO degrees of the $l$--simplices on the bottom complex. This is consistent with the interpretation  of the cross-Betti numbers $\betti{2}{0,-1}$ and $\betti{1}{-1,0}$ given earlier in terms of connectivities between the $1$--skeletons of the horizontal complexes.

\paragraph*{\bf Harmonic cross-hubs.}

Assume for the sake of simplicity that the $X$ is equipped with the trivial weight $\cong 1$. Then, by~\eqref{eq:matrix-representation}, the $(0,0)$--cross-Laplacians $\Lap{T}{0,0}$ and $\Lap{B}{0,0}$ are respectively represented by the $N_{0,0}\times N_{0,0}$--matrices indexed on cross-edges $a_i, a_j\in X_{0,0}$ whose entries are given by
\begin{eqnarray}\label{eq:00LapT}
  (\Lap{T}{0,0})_{a_i,a_j} = \left\{
   \begin{array}{l}
     \deg_{TO}(a_i)+1, \  {\rm if\ } i=j, \\
     \begin{array}{ll}
     1, & {\rm if\ } i\neq j, [a_i]=[v_i^1;v_k^2], \\
     & [a_j]=[v_j^1;v_k^2] \\
     & {\rm and\ } \{v_i^1,v_j^1,v_k^2\}\notin X_{1,0},
     \end{array}
     \\
     
      0, \ {\rm otherwise},
     
   \end{array}
  \right.
\end{eqnarray}

and 
\begin{eqnarray}\label{eq:00LapB}
  (\Lap{B}{0,0})_{a_i,a_j} = \left\{
   \begin{array}{l}
     \deg_{BO}(a_i)+1,  \ {\rm if\ } i=j, \\
     \begin{array}{ll}
     1,  & {\rm if\ } [a_i]=[v_{i_0}^1;v_i^2], \\
     & [a_j]=[v_{i_0}^1;v_j^2], \\
      &  {\rm and\ } \{v_{i_0}^1,v_i^2,v_j^2\}\notin X_{0,1},
      \end{array}
       \\
     
      0, \  {\rm otherwise}.
     
   \end{array}
  \right.
\end{eqnarray}
Applied to the toy example of Fig.~\ref{fig:cross-betti}, $\Lap{T}{0,0}$ has a zero-eigenvalue of multiplicity $3$, generating the three $(0,0)$--cross-cycles in Table~\ref{tab:00generators}. 

\begin{table}\centering

\begin{tabular}{c|c|c|r}
$\omega_1$ & $\omega_2$ & $\omega_3$ & \\
\hline 
 $0.0290$ & - $0.2872$ &    $0.2236$ &  $[  v^1_0  ;   v^2_1 ]$ \\
  $0.0290$ &  - $0.2872$ &  $0.2236$  & $[  v^1_1  ;   v^2_1 ]$ \\
  $0.0290$ &  - $0.28721$ &  $0.2236$ & $[  v^1_2  ;   v^2_1 ]$ \\ 
  $0.0$       &  $0.0$    &    - $0.8944$ &  $[  v^1_4  ;   v^2_1 ]$ \\
  $0.7035$ &  $0.0710$ &   $0.0$       & $[  v^1_4  ;   v^2_4 ]$ \\ 
 - $0.0870$ &  $0.8616$ &  $0.2236$ &  $[  v^1_6  ;   v^2_1 ]$ \\
   $0.0$       &  $0.0$       &  $0.0$        & $[  v^1_6  ;   v^2_2 ]$  \\
  - $0.7035$ & - $0.0710$ &  $0.0$       &$[  v^1_6  ;   v^2_4 ]$ \\
   $0.0$         &  $0.0$      &   $0.0$      &$[  v^1_6  ;   v^2_5 ]$
\end{tabular}%
\caption{Harmonic $(0,0)$--cross-forms. The $3$ eigenvectors of the eigenvalue $0$ of $\Lap{T}{0,0}$ corresponding to the synthetic CSB of Figure~\ref{fig:cross-betti}. There are $2$ harmonic cross-hubs : $v^2_1$ and $v^2_4$, their respective harmonic cross-hubness are $2.6177$ and $1.4070$. }
\label{tab:00generators}
\end{table}

Each coordinate in the eigenvectors is seen as an "intensity" along the corresponding cross-edge. Cross-edges with non-zero intensities sharing the same bottom node define certain communities in the top complex that are "controlled" by the involved bottom node. These community structures depend on both the underlying topology of the top complex and its interdependence with the other complex layer. We then refer to them as {\em harmonic cross-clusters}, and the bottom nodes controlling them are thought of as  {\em harmonic cross-hubs} (HCH). The {\em harmonic cross-hubness} of a bottom node is the $L^1$--norm of the intensities of all cross-edges having it in common. Here, in the eigenvectors of the eigenvalue $0$, there are two subsets of cross-edges with non-zero coordinates: the cross-edges with $v^2_1$ in common, and the ones with $v^2_4$ in common. We therefore have two harmonic cross-hubs (see illustration in Fig.~\ref{fig:sp-xhubs}), hence two harmonic cross-clusters. 
The first one is responsible for the top layer cross-cluster $\{v^1_0,v^1_1, v^1_2,v^1_4, v^1_6\}$, while the second one controls the top layer cross-cluster $\{v^1_4, v^1_6\}$. The intensity of each involved cross-edge is the $L^1$--norm of its corresponding coordinates in the 3 eigenvectors,
and the harmonic cross-hubness is the sum of the intensities of the cross-edges interconnecting the corresponding cross-hub to each of the top nodes in the cross-clusters it controls. For instance, $v^2_1$ is the bottom node with the highest harmonic cross-hubness which is $2.6177$. This reflects the fact $v^2_1$ not only interconnects the two connected components of the top complex (which $v^2_4$ does as well), but it also allows fast-track connections between the highest number of nodes that are not directly connected with intra-layer edges in the top complex. The same calculations applied to the eigenvectors of the zero-eigenvalues of $\Lap{B}{0,0}$ give $v^1_6$ as the top node with the highest harmonic cross-hubness \wrt  the bottom complex. 


%
%
%
%

%

\paragraph*{\bf Spectral persistence of cross-hubs.}

\begin{figure}[!h]
\begin{center}
\begin{tabular}{ll}
(A) & (B) \\
\includegraphics[width=0.22\textwidth]{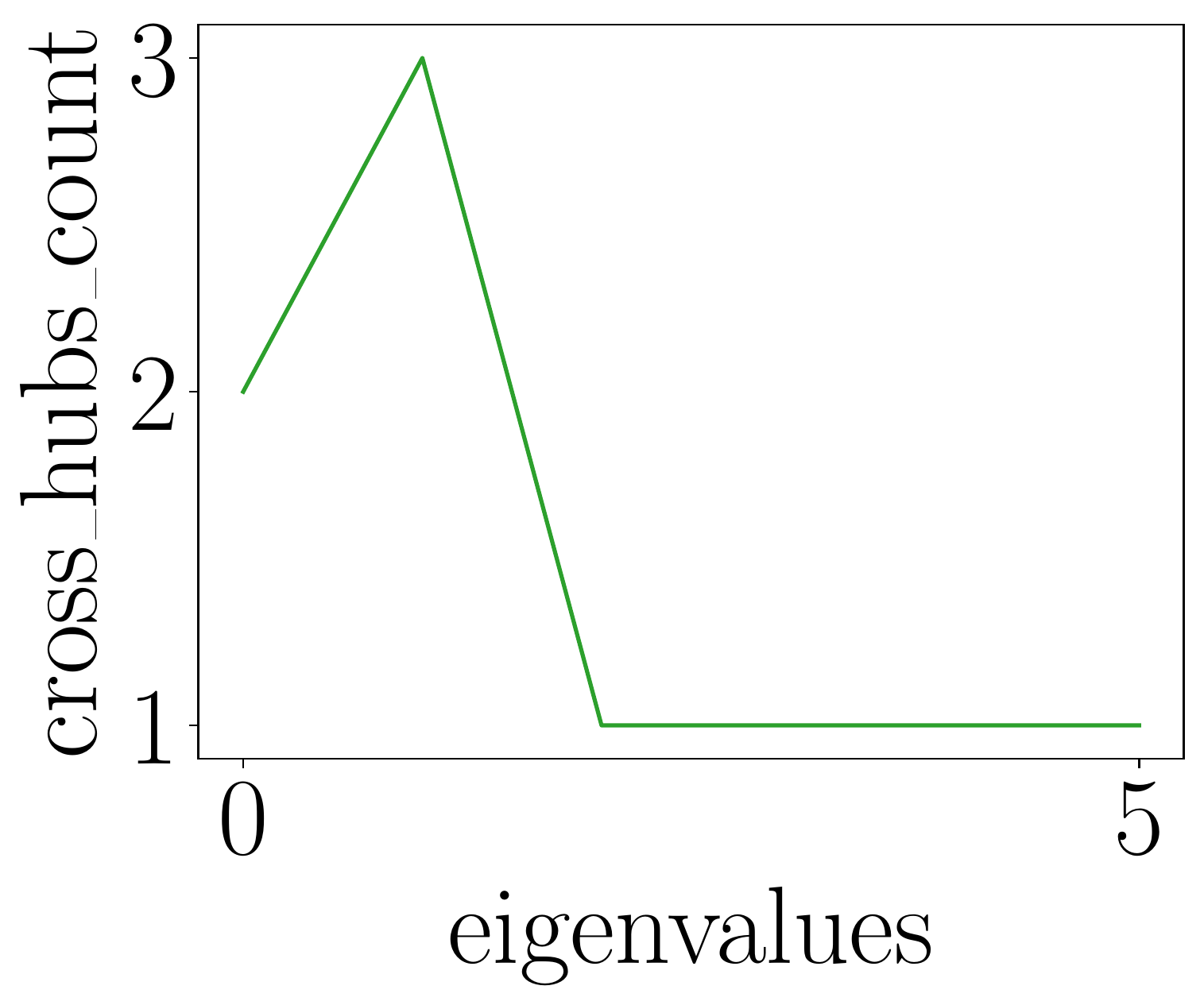} 
& \includegraphics[width=0.22\textwidth]{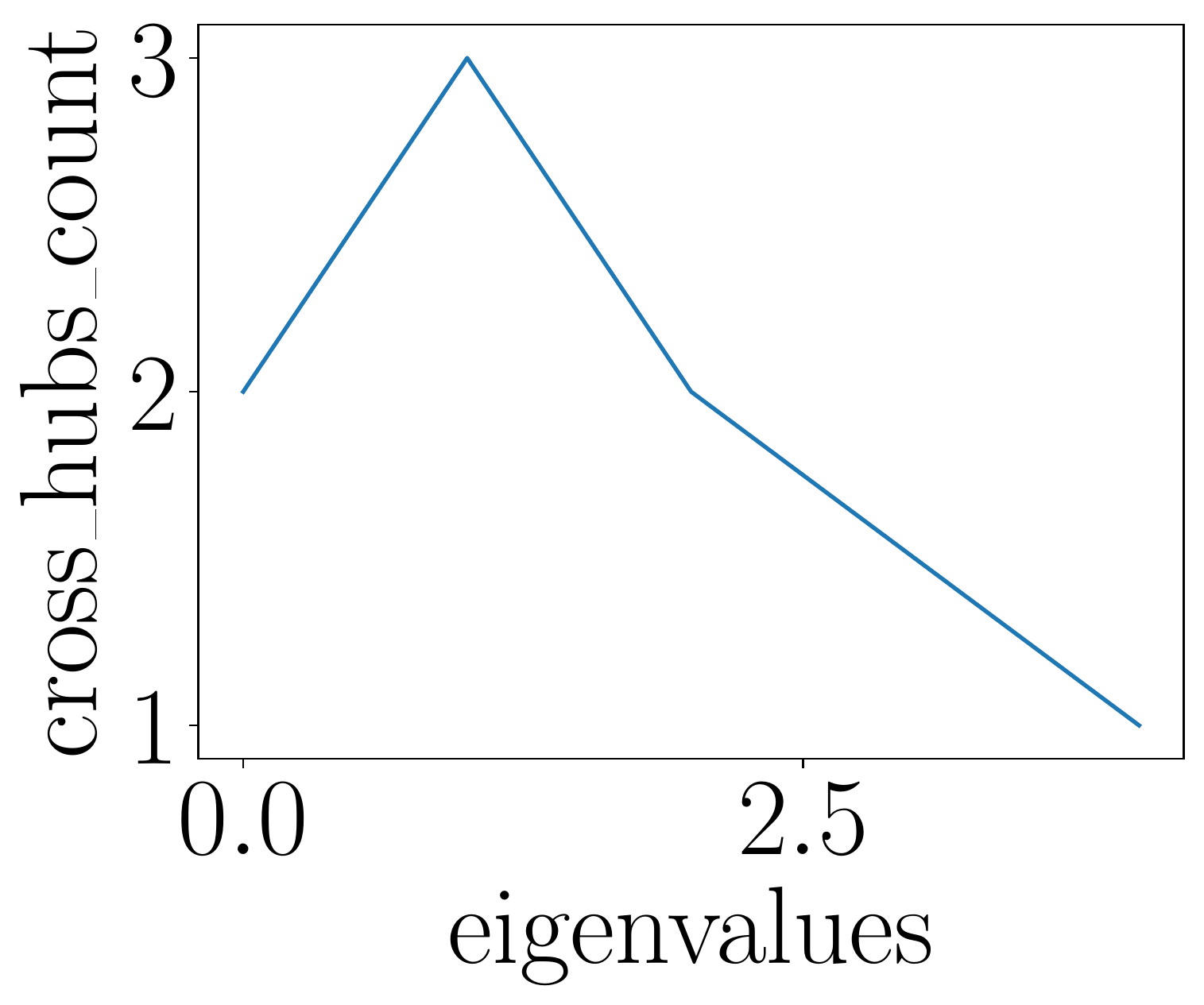} \\
(C) & (D) \\
 \includegraphics[width=0.22\textwidth]{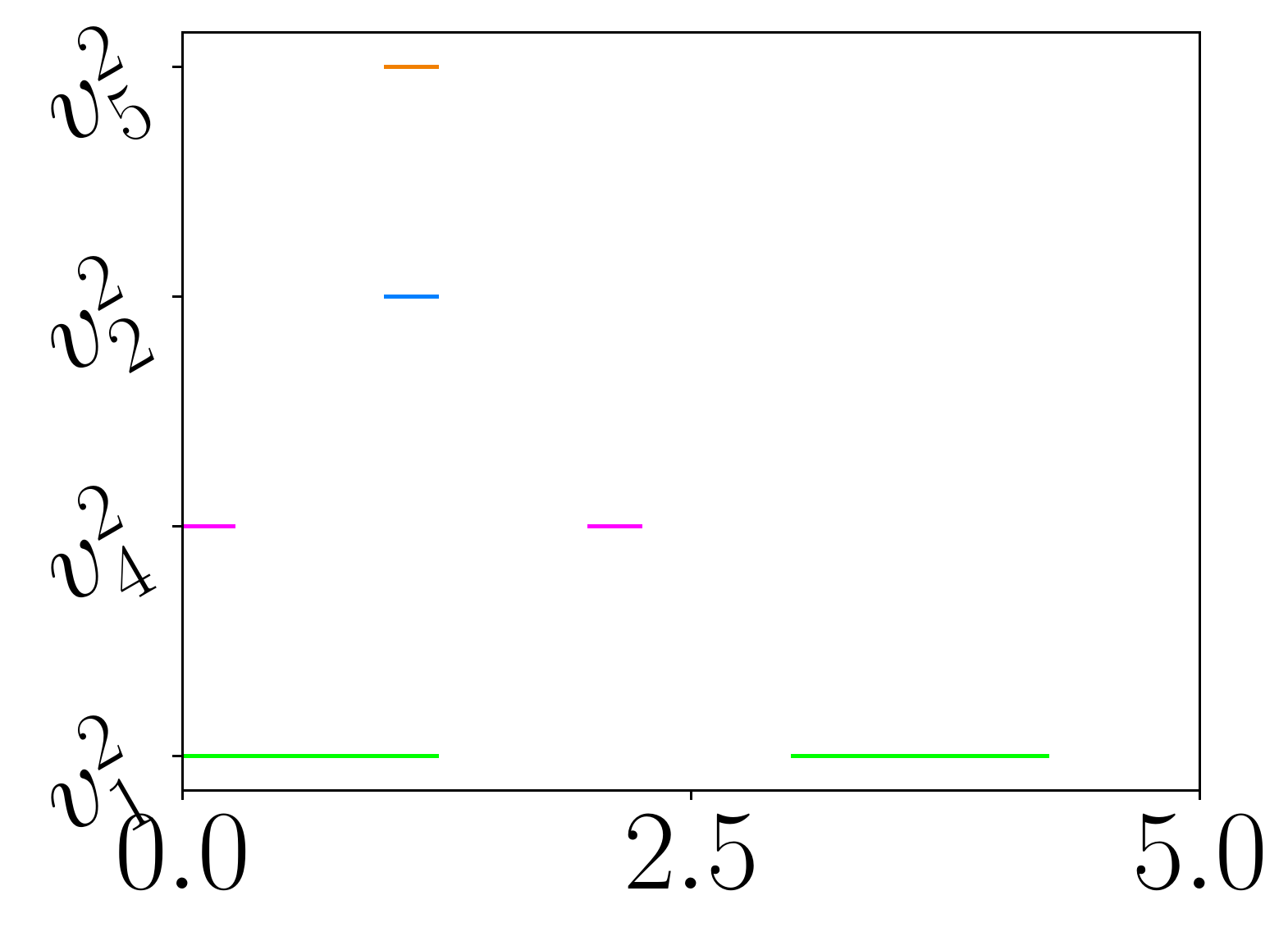} 
&
 \includegraphics[width=0.22\textwidth]{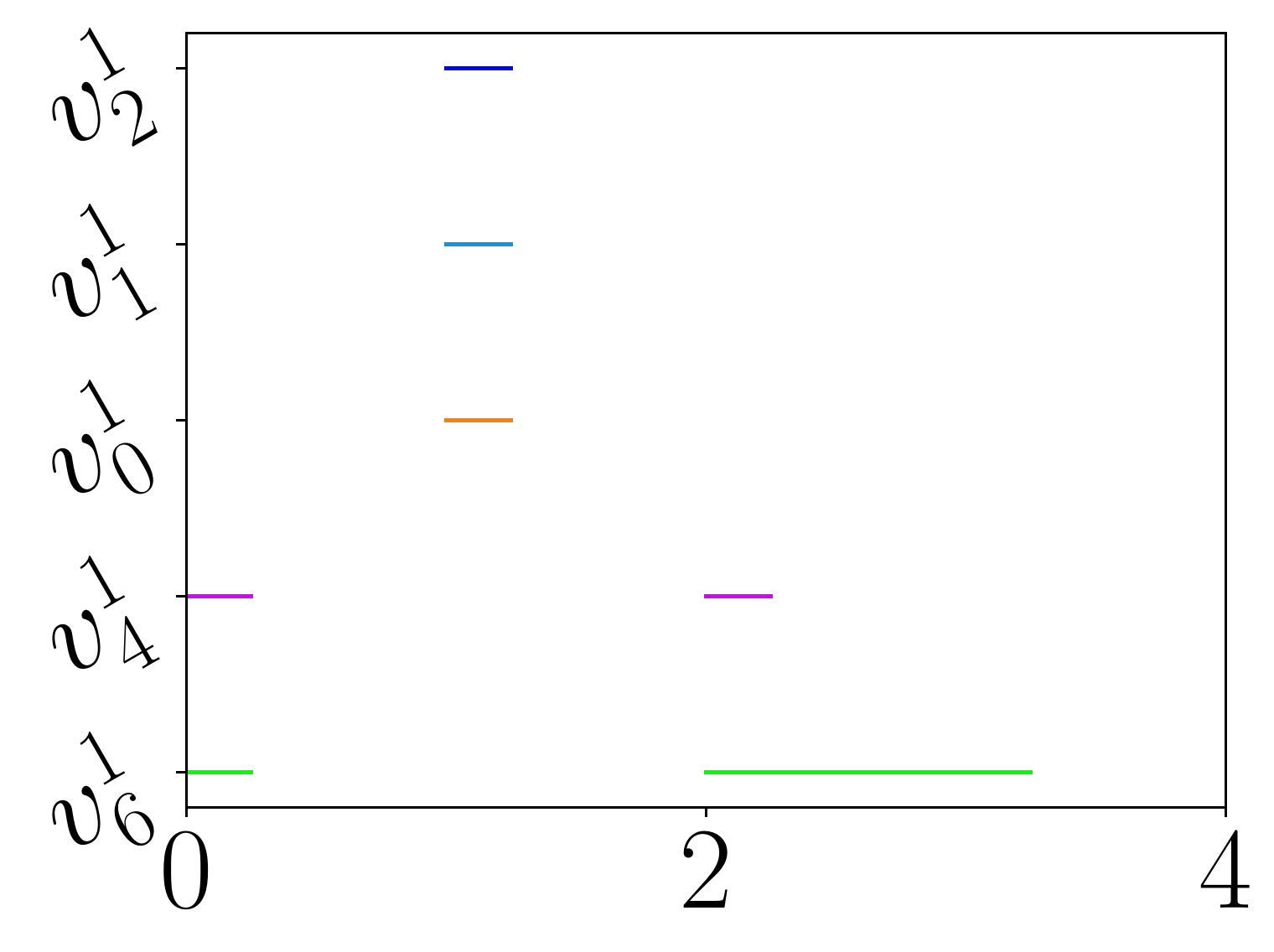}
\end{tabular}
 \end{center}
\caption{{\bf Spectral persistence of cross-hubs.} Schematic illustrations of the variations of spectral cross-hubs along the eigenvalues and  the spectral persistence bars codes for the toy CSB of Fig.~\ref{fig:cross-betti}: (A) shows the number of bottom nodes that emerge as spectral cross-hubs \wrt the top layer as a function of the eigenvalues of $\Lap{T}{0,0}$, and (B) represents the number of top nodes revealed as spectral cross-hubs \wrt the bottom layer as a function of the eigenvalues of $\Lap{B}{0,0}$. (C) and (D) represent the {\em spectral persistence bar codes} for $\Lap{T}{0,0}$ and $\Lap{B}{0,0}$, respectively. For both the top and bottom $(0,0)$--cross-Laplacians, most of the spectral cross-hubs, hence of spectral cross-clusters, emerge during the first stages (smallest eigenvalues), very few of them survive at later stages, and here only one cross-hub emerge or survive at the largest eigenvalue ($v^2_1$ for $\Lap{T}{0,0}$ and $v^1_6$ for $\Lap{B}{0,0}$).}
\label{fig:pers-bars}
\end{figure}

To better grasp the idea of cross-hubness, let us have a closer look at the coordinates of the eigenvectors of the $(0,0)$--cross-Laplacians (\eqref{eq:00LapT} and~\eqref{eq:00LapB}) whose eigenvalues are all non-negative real numbers. Suppose $\phi=(x_1, ..., x_{N_{0,0}})$ is an eigenvector for an eigenvalue $\l^T$ of $\Lap{T}{0,0}$. Then, denoting the cross-edges by $a_i, i=1, ..., N_{0,0}$, we have the relations 
\begin{eqnarray}
x_i = \frac{1}{\l^T-\deg_{TO}(a_i)}\sum_j \chi(a_i,a_j)x_j,
\end{eqnarray}
where $\chi$ is such that $\chi(a_i,a_j)=1$ if $i=j$ or if $a_i$ and $a_j$ are adjacent but do not belong to a top cross-triangle, and $\chi(a_i,a_j)=0$ otherwise. It follows that the cross-edge intensity $|x_i|$ grows larger as $\deg_{TO}(a_i)\to \l^T$. In particular, for $\l^T=0$, the intensity is larger for cross-edges that belong to a large number of cones and to the smallest number of top cross-triangles. Now, consider the other extreme of the spectrum, namely $\l^T=\l^T_{max}$ to be the largest eigenvalue of $\Lap{T}{0,0}$. Then, the intensity $|x_i|$ is larger for cross-edges belonging to the largest number of top cross-triangles and a large number of top cones at the same time.

\begin{figure*}[!ht]
\begin{center}
\begin{tabular}{lll}
(A) &  (B) & (C)\\
\includegraphics[width=0.20\textwidth]{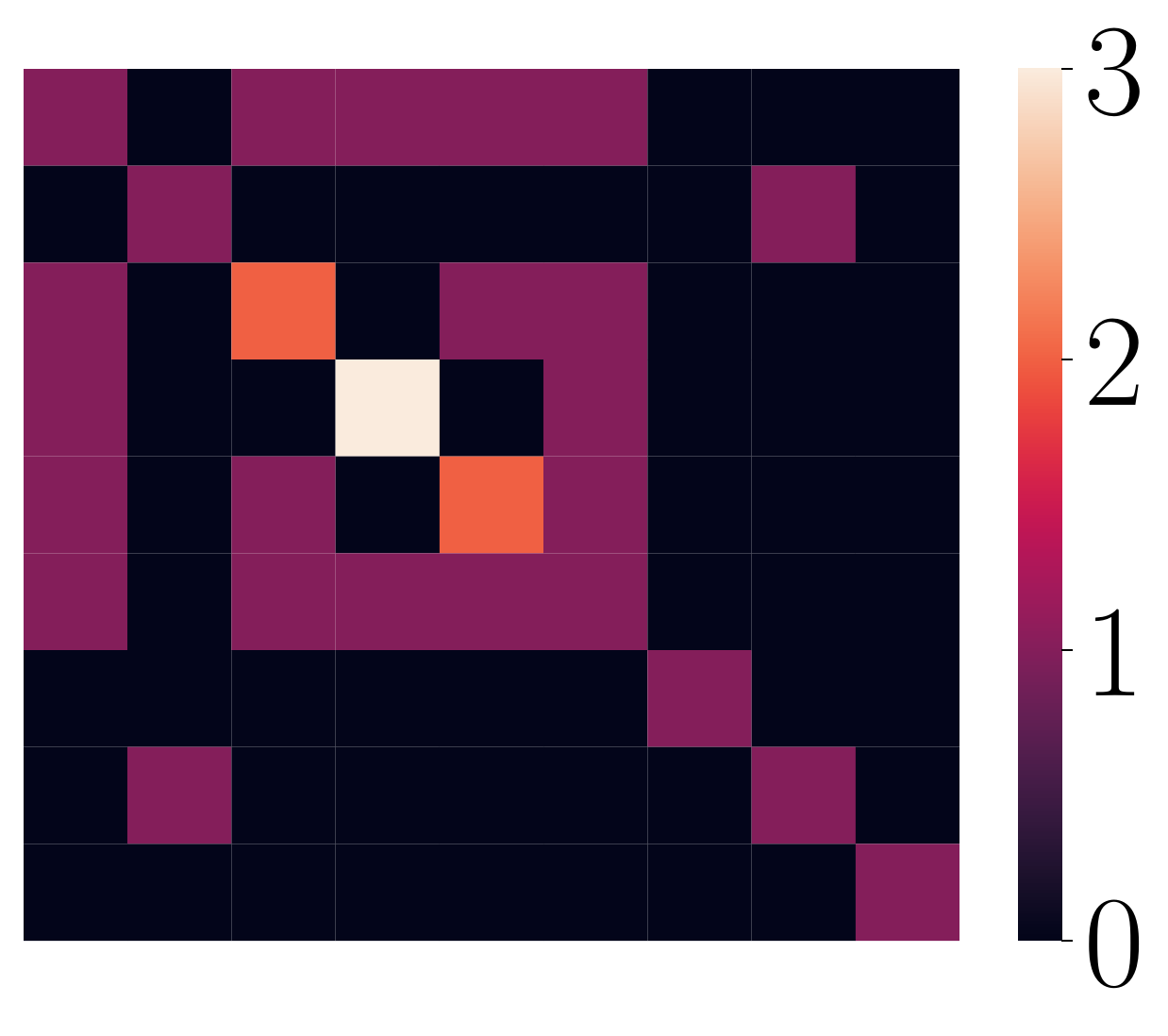}
 &  
 \includegraphics[width=0.24\textwidth]{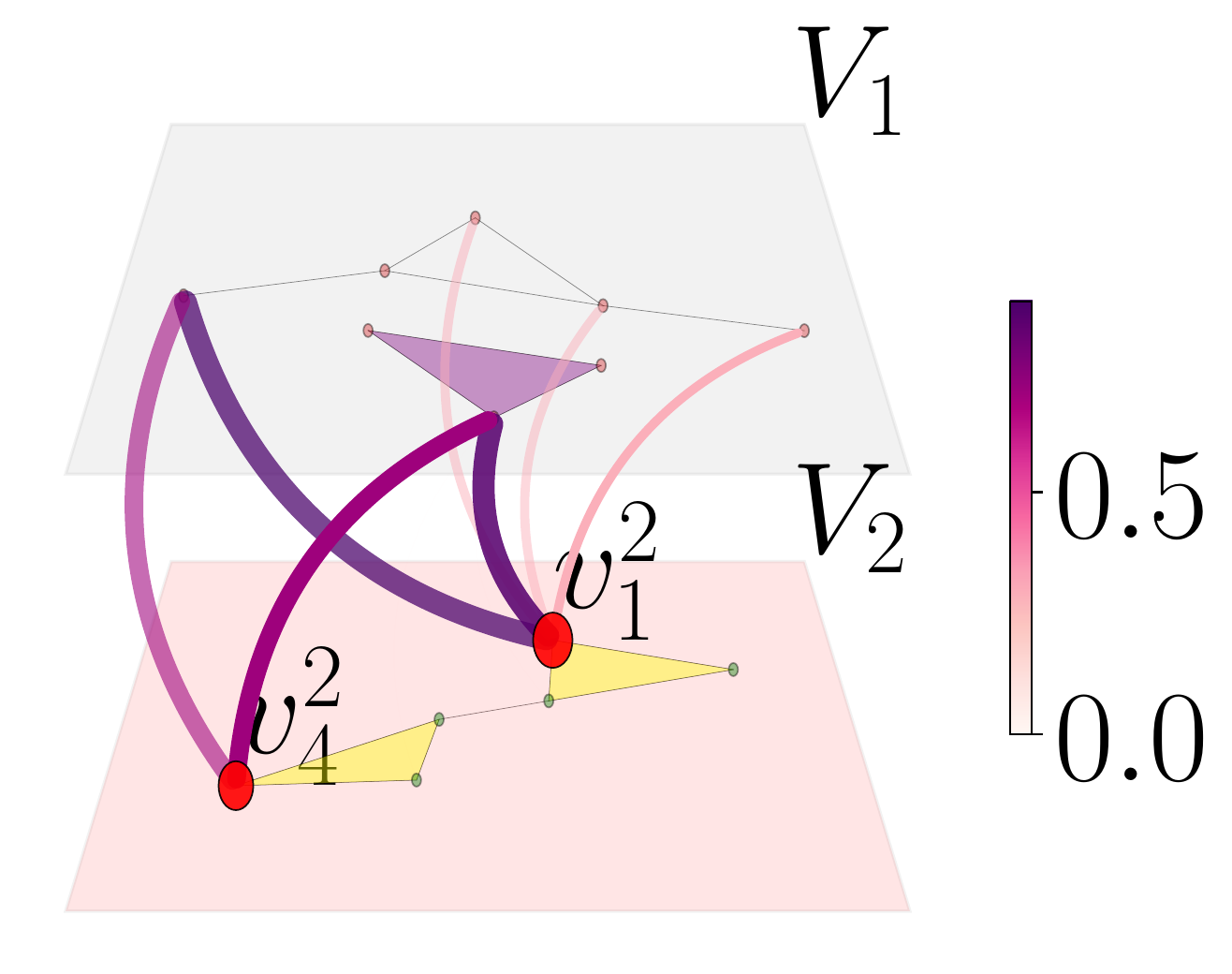} 
& 
\includegraphics[width=0.24\textwidth]{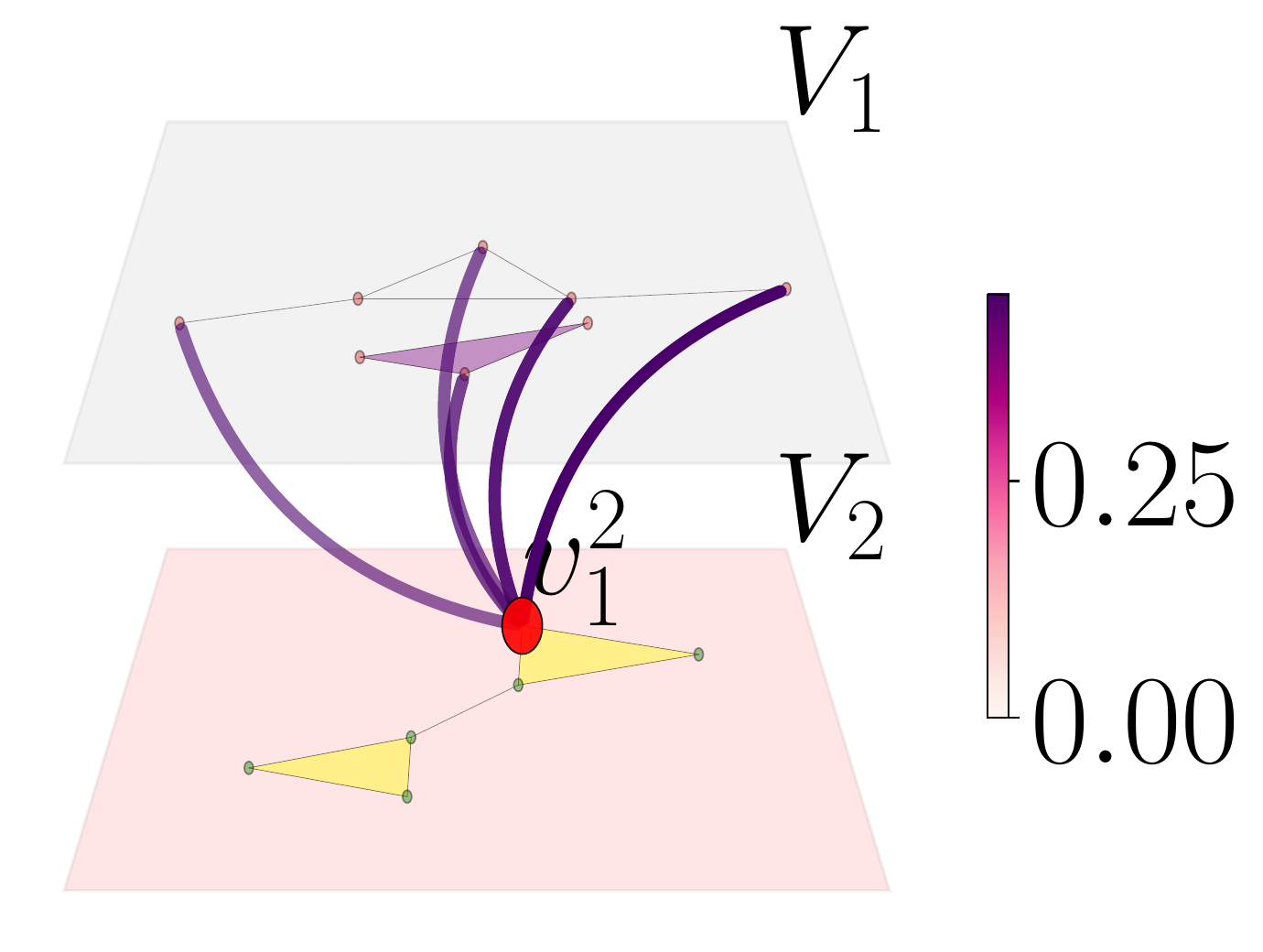}  \\  
(D) & (E) & (F) \\
\includegraphics[width=0.20\textwidth]{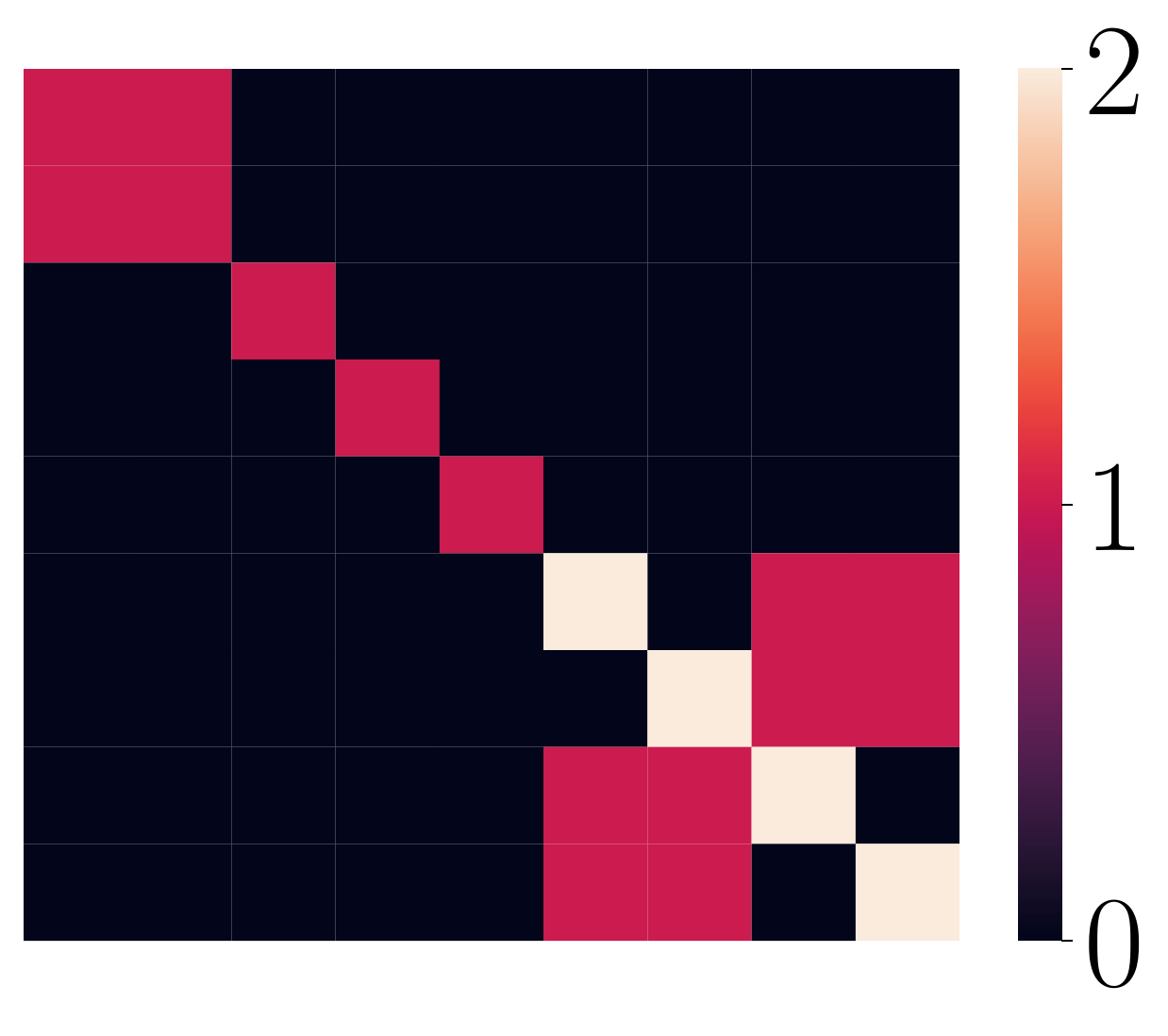} & 
\includegraphics[width=0.24\textwidth]{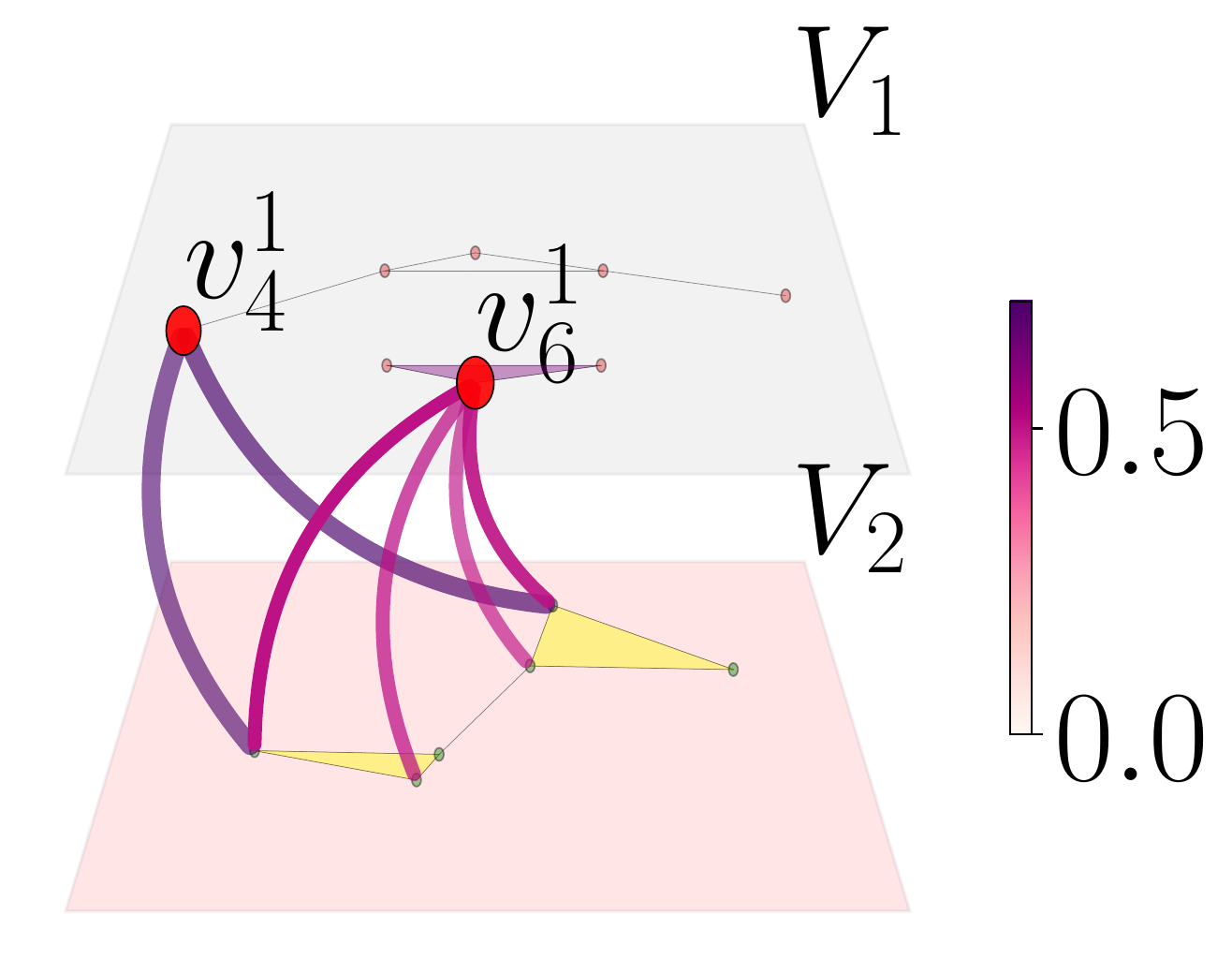} 
& 
\includegraphics[width=0.24\textwidth]{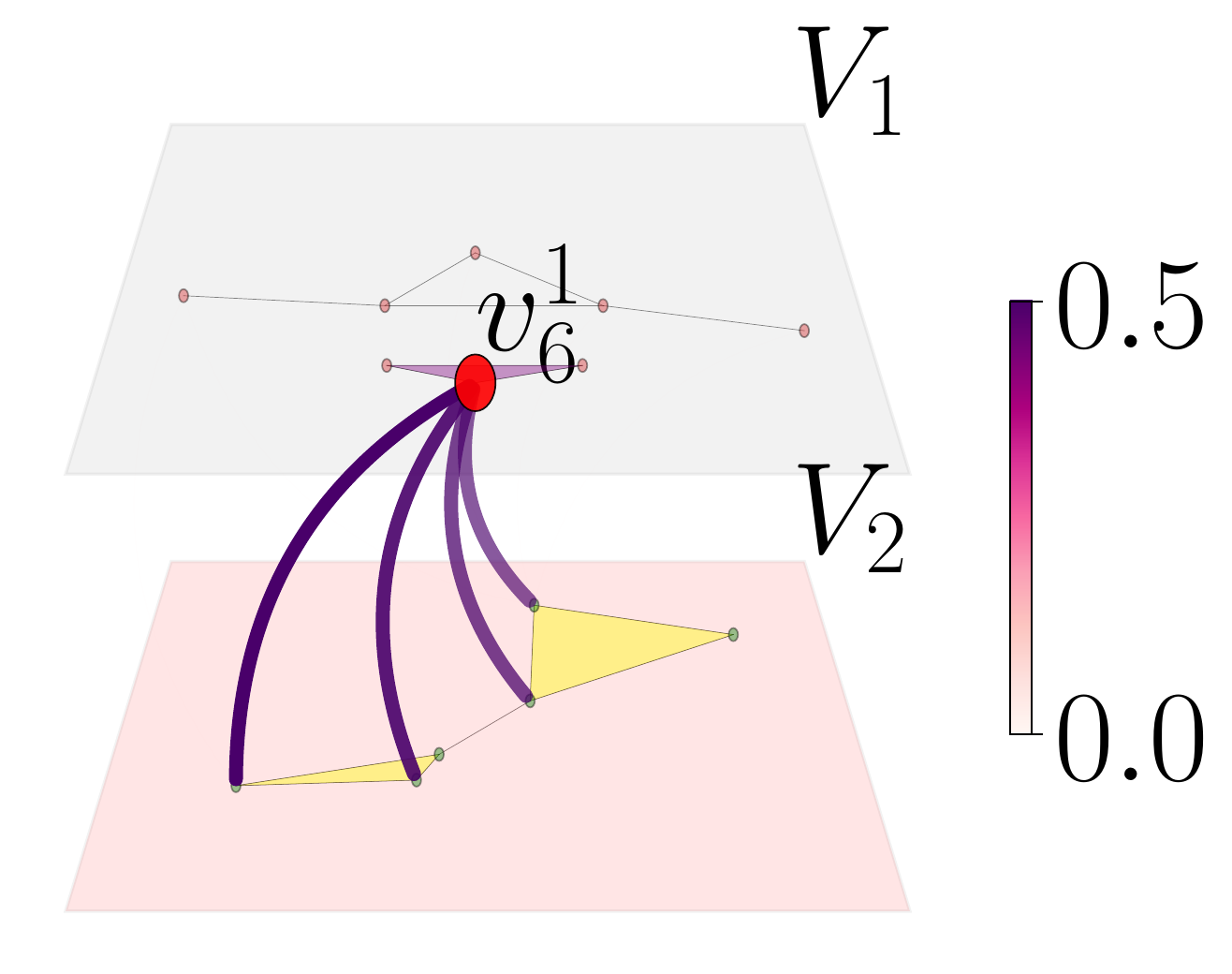} \\

\end{tabular}
\end{center}

\caption{{\bf Cross-Laplacians, harmonic and principal cross-hubs.} (A) and (D): Heat-maps of the the top and bottom $(0,0)$--cross-Laplacian matrices for the example of Fig.~\ref{fig:cross-betti}. Both matrices are indexed over the cross-edges of the CSB, and the diagonal entries correspond to one added to the number of cross-triangles containing the corresponding cross-edge. $\Lap{T}{0,0}$ has a zero eigenvalue of multiplicity $3$, while $\Lap{B}{0,0}$ has a zero eigenvalue of multiplicity $2$. (B) and (E): The harmonic cross-hubs \wrt to the top (resp. the bottom) horizontal complex of $X$; the {\em intensity} of a cross-edge is given by the $L^1$--norm of the corresponding coordinates in the eigenvectors of the eigenvalue $0$. (C) and (F): the {\em principal cross-hubs} in the bottom (resp. top) layer \wrt the top (resp. bottom) layer; by definition, they are the spectral cross-hubs obtained from by the largest eigenvalues of the top and bottom $(0,0)$--cross-Laplacians, respectively.}
\label{fig:sp-xhubs}
\end{figure*}

Taking the case of a 2-layered network, for $\l^T=0$, $|x_i|$ is larger for a cross-edge pointing to a bottom node interconnecting a largest number of top nodes that are not directly connected with intra-layer edges; and for $\l^T=\l_{max}$, $|x_i|$ is larger for a cross-edge pointing to a bottom node interconnecting a large number of top intra-layer communities both with each other and with a large number of top nodes that are not directly connected to each other via intra-layer edges.

 More generally, by applying the same process to each distinct eigenvalue, we obtain clustering structures in the top layer that are {\em controlled} by the bottom nodes and that vary along the spectrum $\l_1^T\le \l_2^T\le \cdots \le \l^T_{max}$ of $\Lap{T}{0,0}$. At every  stage, we regroup the cross-edges with non-zero coordinates in the associated eigenvectors and pointing to the same nodes, then sum up their respective intensities to obtain a ranking among a number of cross-hubs that we call {\em spectral cross-hubs} (SCHs). Intuitively, the intensities held by cross-edges gather to confer a 'restructuring power' onto the common bottom node -- the cross-hub, allowing it to control a cluster on the top layer. It is clear that, by permuting the top layer with the bottom layer, the same reasoning applies to $\Lap{B}{0,0}$. In particular, we define the {\em principal cross-hubs} (PCH) in the bottom layer \wrt the top layer as the SCHs obtained from $\l^T_{\max}$. The {\em principal cross-hubness} of a bottom PCH is defined as its restructuring power. In a similar fashion, we define the principal cross-hubness in the top layer \wrt the bottom layer using the largest eigenvalue $\l^B_{max}$ of $\Lap{B}{0,0}$. 
Going back to the bicomplex of Fig.~\ref{fig:cross-betti}, the largest eigenvalue of $\Lap{T}{0,0}$ is $\l^T_{max}=5$, the corresponding eigenvector is represented by Table~\ref{tab:CSB-PCH}.

\begin{table}[!ht]
\centering
\begin{tabular}{c|c}
    $0.4472$ &  $[  v^1_0  ;   v^2_1 ]$ \\
    $0.4472$  & $[  v^1_1  ;   v^2_1 ]$ \\
   $0.4472$ & $[  v^1_2  ;   v^2_1 ]$ \\ 
  $0.4472$ &  $[  v^1_4  ;   v^2_1 ]$ \\
    $0.0$       & $[  v^1_4  ;   v^2_4 ]$ \\ 
   $0.4472$ &  $[  v^1_6  ;   v^2_1 ]$ \\
    $0.0$        & $[  v^1_6  ;   v^2_2 ]$  \\
   $0.0$       &$[  v^1_6  ;   v^2_4 ]$ \\
   $0.0$      &$[  v^1_6  ;   v^2_5 ]$
 \end{tabular}
 \caption{Principal eigenvector of $\Lap{T}{0,0}$ for the CSB of Figure~\ref{fig:cross-betti}. By definition, this is the eigenvector associated to the largest eigenvalue.}
 \label{tab:CSB-PCH}
\end{table}

There is only one PCH in the bottom layer \wrt the top layer, which is the bottom node $v^2_1$, and its principal cross-hubness is $2.2360$.

Interestingly, the number of SCHs that appear for a given eigenvalue tend to vary dramatically \wrt the smallest eigenvalues before it eventually decreases or stabilizes at a very low number (see Fig.~\ref{fig:pers-bars} and Fig.~\ref{fig:atn-persistence-SPB}). Some cross-hubs may appear at one stage along the spectrum and then disappear at a future stage. This suggests the notion of {\em spectral persistence} of cross-hubs. Nodes that emerge the most often or live longer as cross-hubs along the spectrum might be seen as the most central in restructuring the topology of the other complex layer.  The more we move far away from the smallest non-zero eigenvalue, the most powerful are he nodes that emerge as hubs facilitating communications between aggregations of nodes in the other layer. The emergence of spectral cross-hubs is represented by a horizontal line -  {\em spectral persistence bar} -  running through the indices of the corresponding eigenvalues (Fig.~\ref{fig:pers-bars}). The spectral persistence bars corresponding to all SCHs (the {\em spectral bar codes}) obtained from $\Lap{T}{0,0}$ (resp. $\Lap{B}{0,0}$) constitute a {\em signature} for all the clustering structures imposed by the bottom (resp. top) layer to the top (resp. bottom) layer.



\begin{figure*}[!th]
\begin{tabular}{cc}
Lufthansa to Ryanair & Lufthansa to Easyjet \\
 \includegraphics[width=0.42\textwidth]{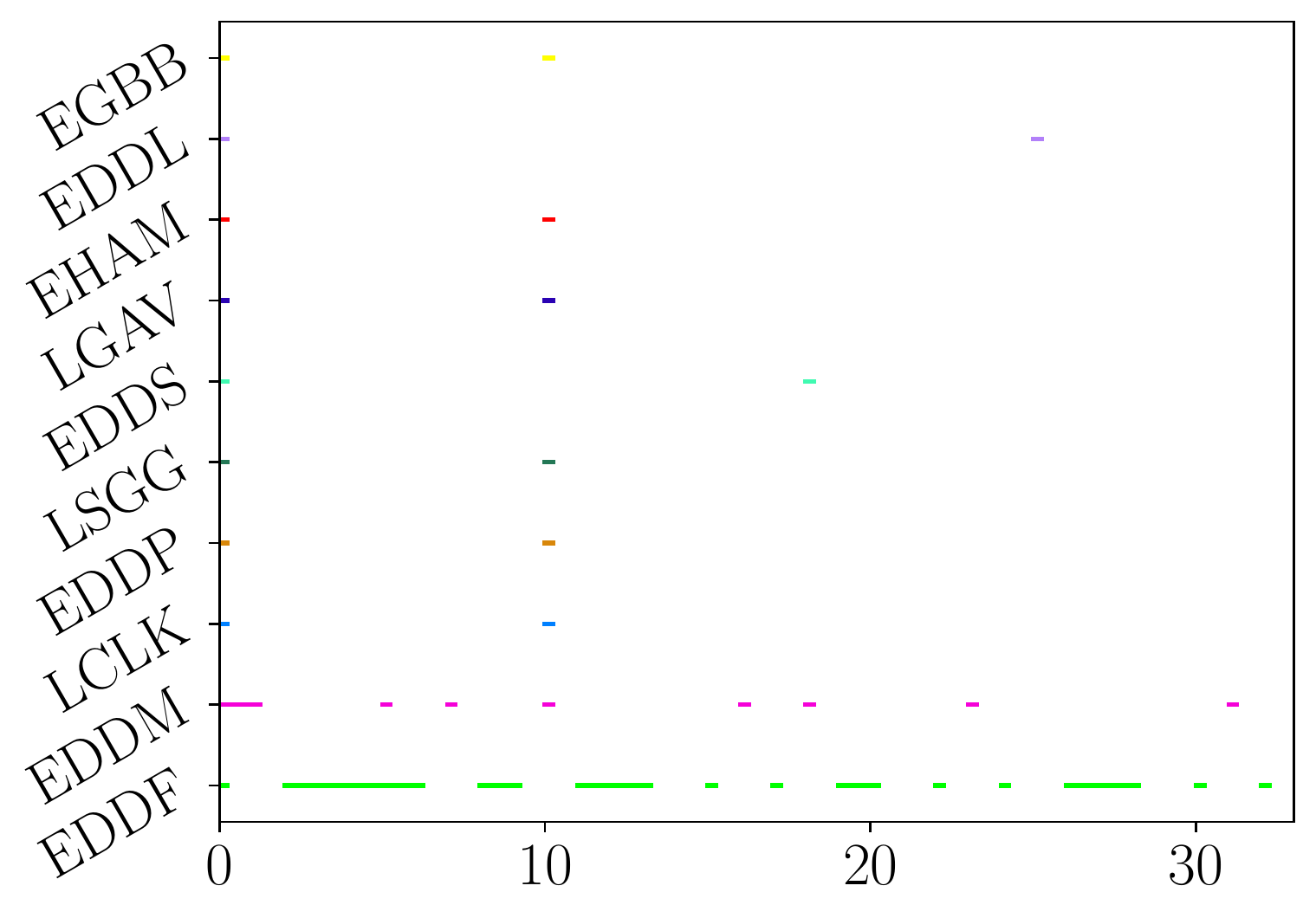} 
 &
  \includegraphics[width=0.42\textwidth]{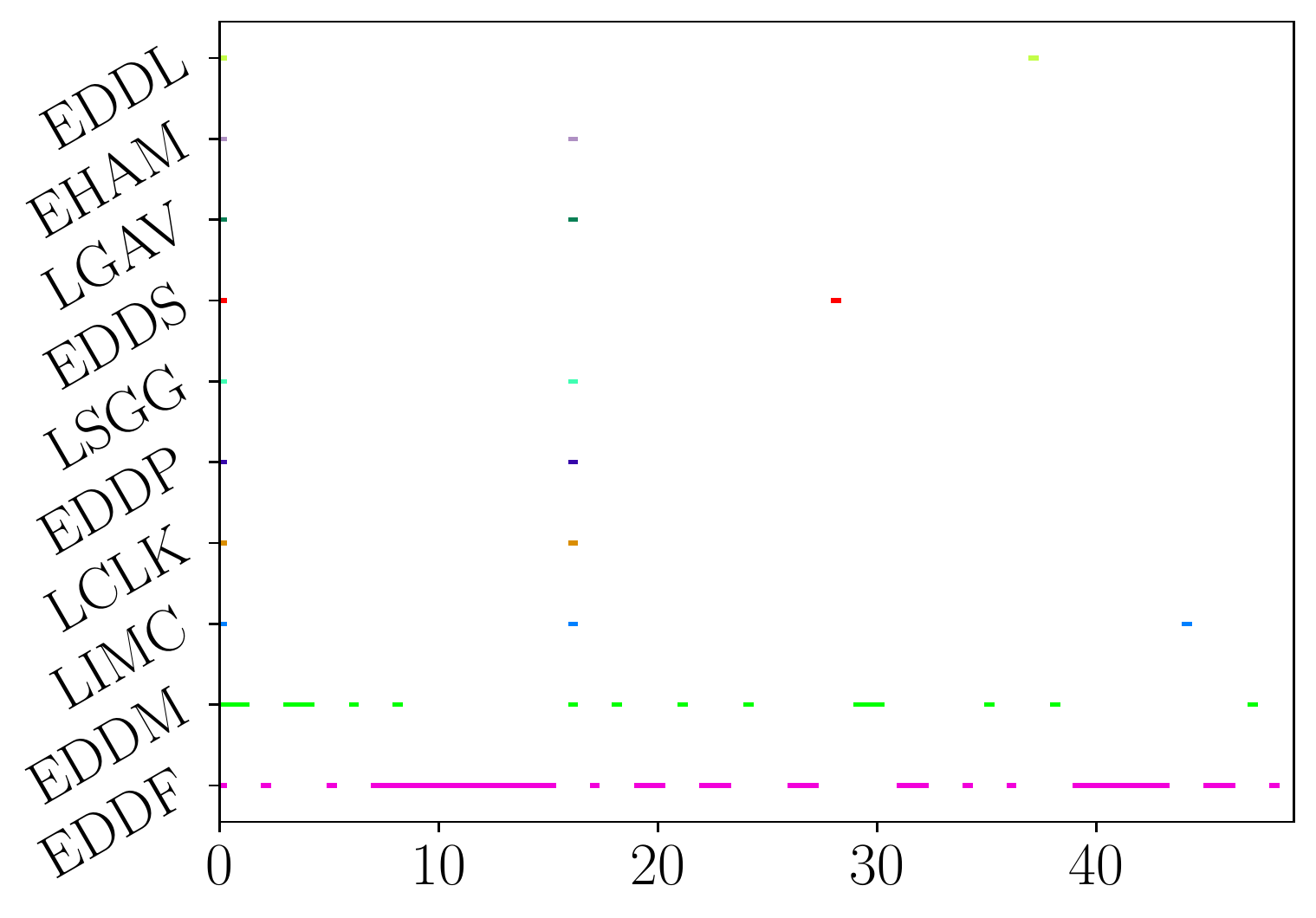}  \\
  Ryanair to Lufthansa & Rynair to Easyjet \\
    \includegraphics[width=0.42\textwidth]{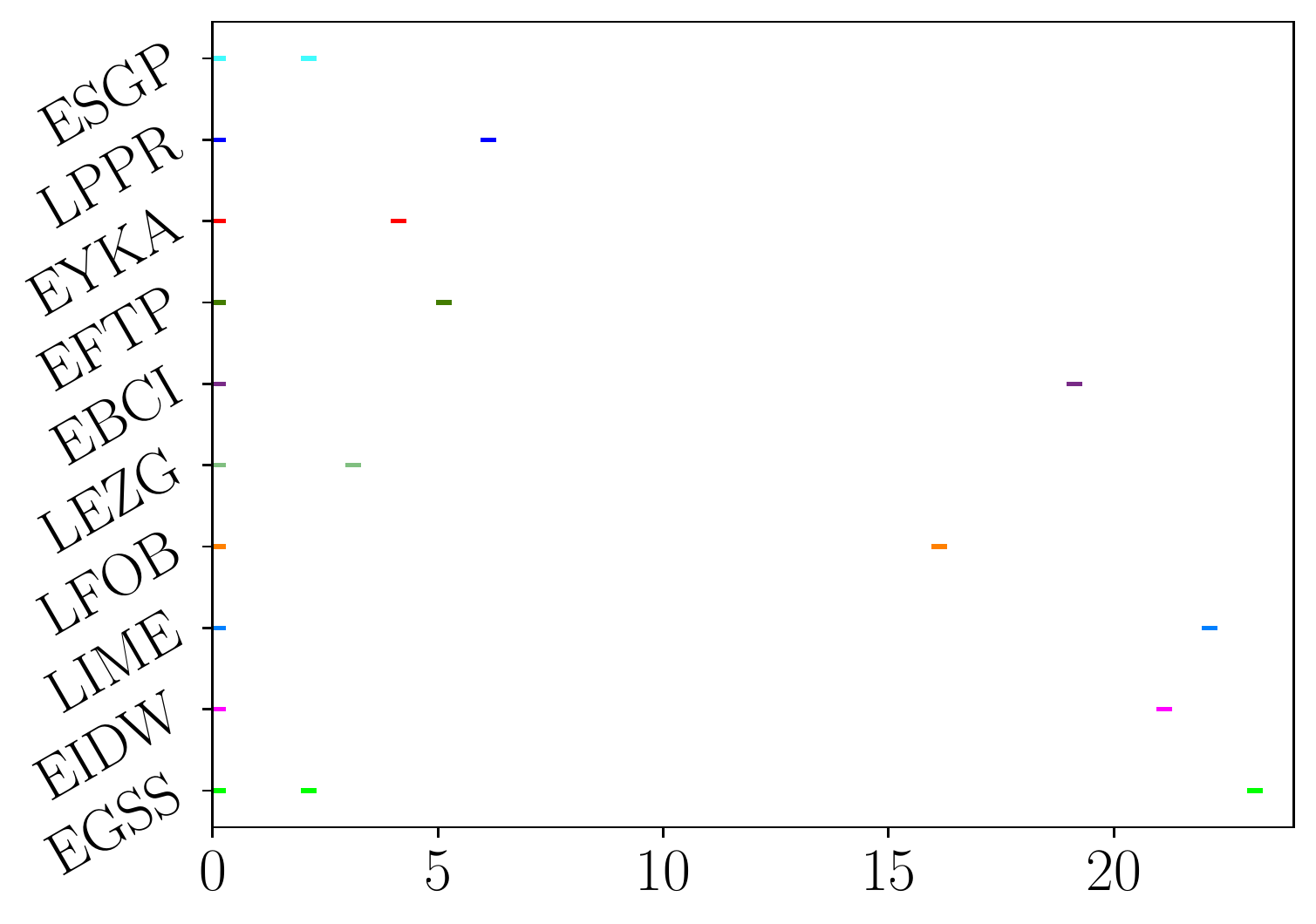}  
    &
    \includegraphics[width=0.42\textwidth]{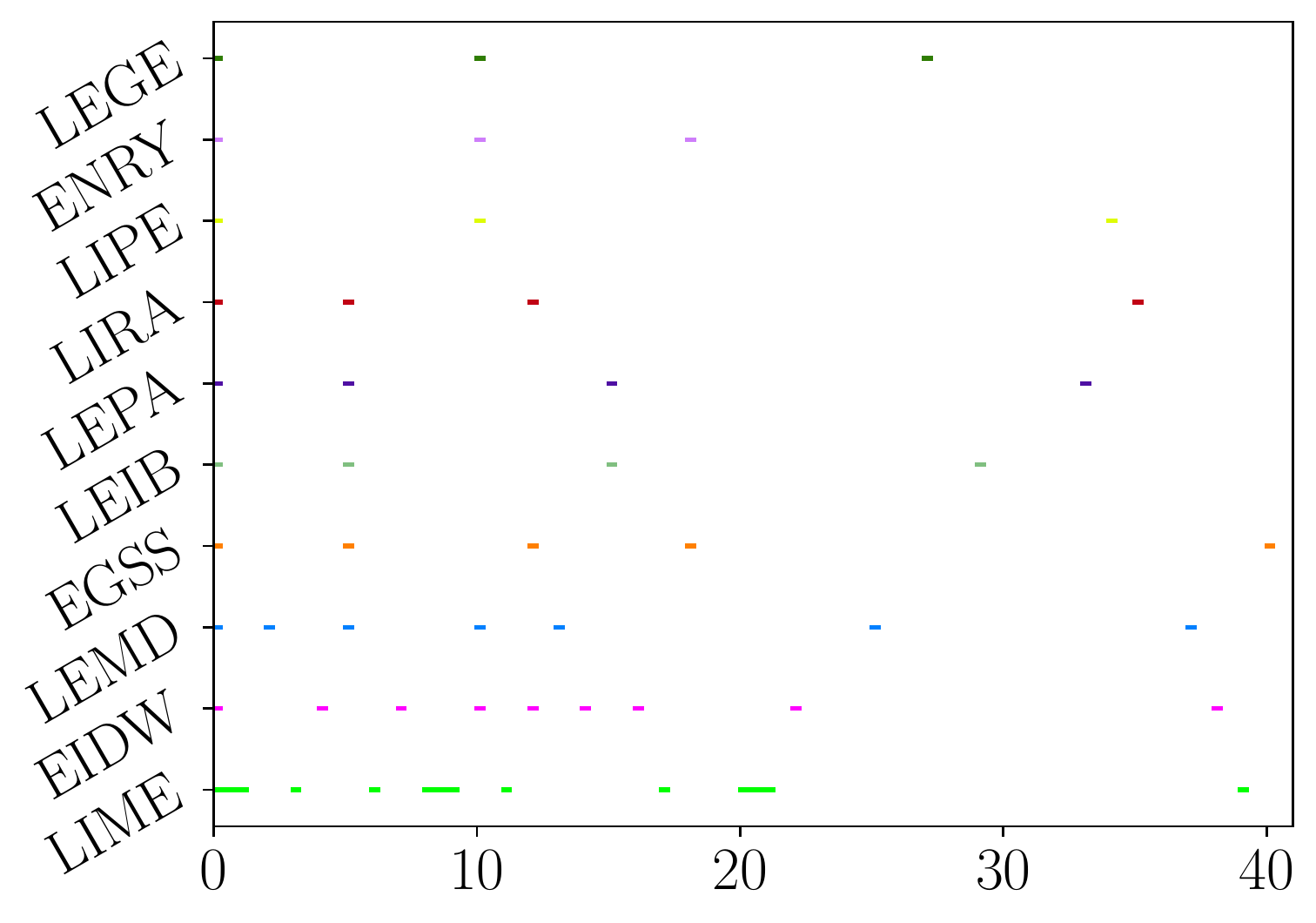} \\
 Easyjet to Lufthansa & Easyjet to Ryanair \\ 
  \includegraphics[width=0.42\textwidth]{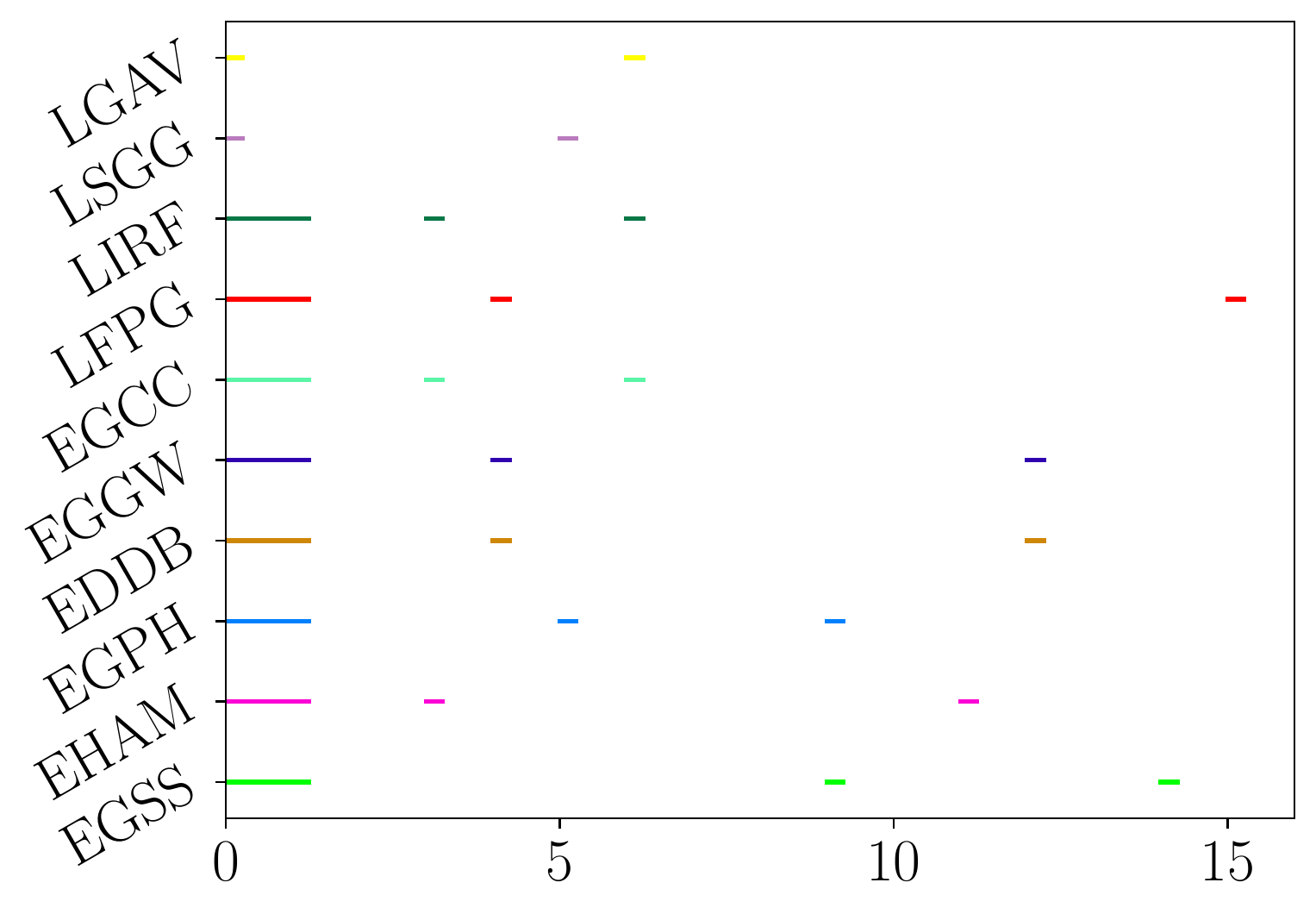}  
  &
    \includegraphics[width=0.42\textwidth]{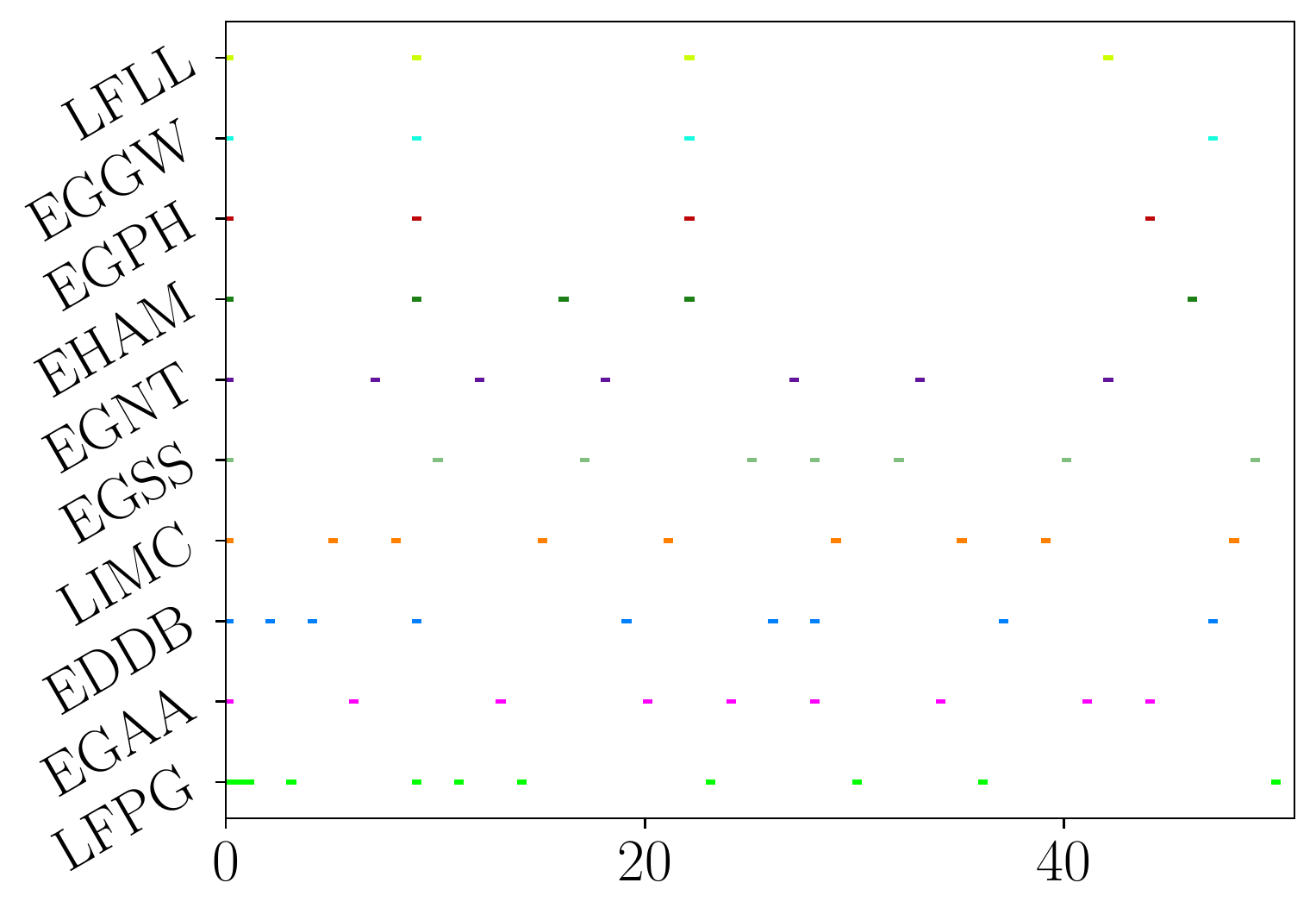}  

\end{tabular}
\caption{{\bf Spectral persistent cross-hubs.}  The spectral persistence bar codes of the six diffusion bicomplexes of the European ATN multiplex. The nodes represent European airports labelled with their ICAO codes. The most persistent cross-hubs correspond to the airports that provide the most efficient correspondences from the first airline network to the second.}
\label{fig:atn-persistence-SPB}
\end{figure*}

\section{Experiments on multiplex networks}
\paragraph*{\bf Diffusion CSBs.}
Let $\cal M$ be a multiplex formed by $M$ graphs $\Gamma^s = (E_s,V), s=1,\ldots,M$. Denoting the vertex set $V$ as an ordered set $\{1,2,\ldots, N$\}, we will write $v^s_i$ to represent the node $i$ in the graph $\Gamma^s$, following the same notations we have used for multicomplexes. 

For every pair of distinct indices $s,t$, we define the $2$--dimensional CSB $\DiffBicomp{X}{s}{t}$ on $V\times V$ such that
$\DiffBicomp{X}{s}{t}_{k,-1} = \emptyset$ for $k\ge 1$,  $\DiffBicomp{X}{s}{t}_{-1,k}$ is the $2$--clique complex of the layer indexed by $t$ in the multiplex $\cal M$; a pair $(v_i^s, v_j^t)\in V\times V$, forms a cross-edge if $i<j$, and nodes $i$ and $j$ are connected in $\Gamma^s$; and
 a $(0,1)$--crossimplex is a triple $(v_i^s, v_j^t, v_k^t)\in V^3$ such that $i$ is connected to $j$ and $k$ in $\Gamma^s$, and $j$ and $k$ are connected in $\Gamma^t$, while  $\DiffBicomp{X}{s}{t}_{1,0}=\emptyset$. We call $\DiffBicomp{X}{s}{t}$ the {\em diffusion bicomplex} of (layer) $s$ onto $t$. Notice that by construction, the $(0,0)$--cross-Laplacians of $\DiffBicomp{X}{s}{t}$ are indexed over $E_s$, while the $(0,0)$--cross-Laplacians of $\DiffBicomp{X}{t}{s}$ are indexed over $E_t$. This shows that $\DiffBicomp{X}{s}{t}$ and $\DiffBicomp{X}{t}{s}$ are not the same. In fact, the diffusion bicomplex $\DiffBicomp{X}{s}{t}$ is a way to look at the the topology of $\Gamma^s$ through the topology of $\Gamma^t$; or put differently, it diffuses the topology of the former into the topology of the latter.

\paragraph*{\bf Cross-hubs in air transportation networks.}

We use a subset of the European Air Transportation Network (ATN) dataset from~\cite{cardillo2013emergence} to construct a $3$--layered multiplex $\cal M$ on $450$ nodes 
each representing a European airport~\cite{wu2019tensor}. The $3$ layer networks $\Gamma^1, \Gamma^2$, and $\Gamma^3$ of $\cal M$ represent the direct flights served by Lufthansa, Ryanair, and Easyjet airlines, respectively; that is, intra-layer edges correspond to direct flights between airports served by the corresponding airline. Considering the respective bottom $(0,0)$--cross-Laplacians of the six diffusion bicomplexes $\DiffBicomp{X}{1}{2}$, $\DiffBicomp{X}{1}{3}$, $\DiffBicomp{X}{2}{1}$, $\DiffBicomp{X}{3}{1}$, $\DiffBicomp{X}{2}{3}$, and $\DiffBicomp{X}{3}{2}$, we obtain the spectral persistence bar codes describing the emergence of SCH's for each airline \wrt the others (see Fig.~\ref{fig:atn-persistence-SPB}). The induced SCH rankings are presented in Table~\ref{tab:EATN-rankings}, while the corresponding PCHs are illustrated in Fig.~\ref{fig:atn-SPB}.

\begin{figure*}[!ht]
\begin{tabular}{cc}
Lufthansa to Ryanair & Lufthansa to Easyjet \\
 \includegraphics[width=0.42\textwidth]{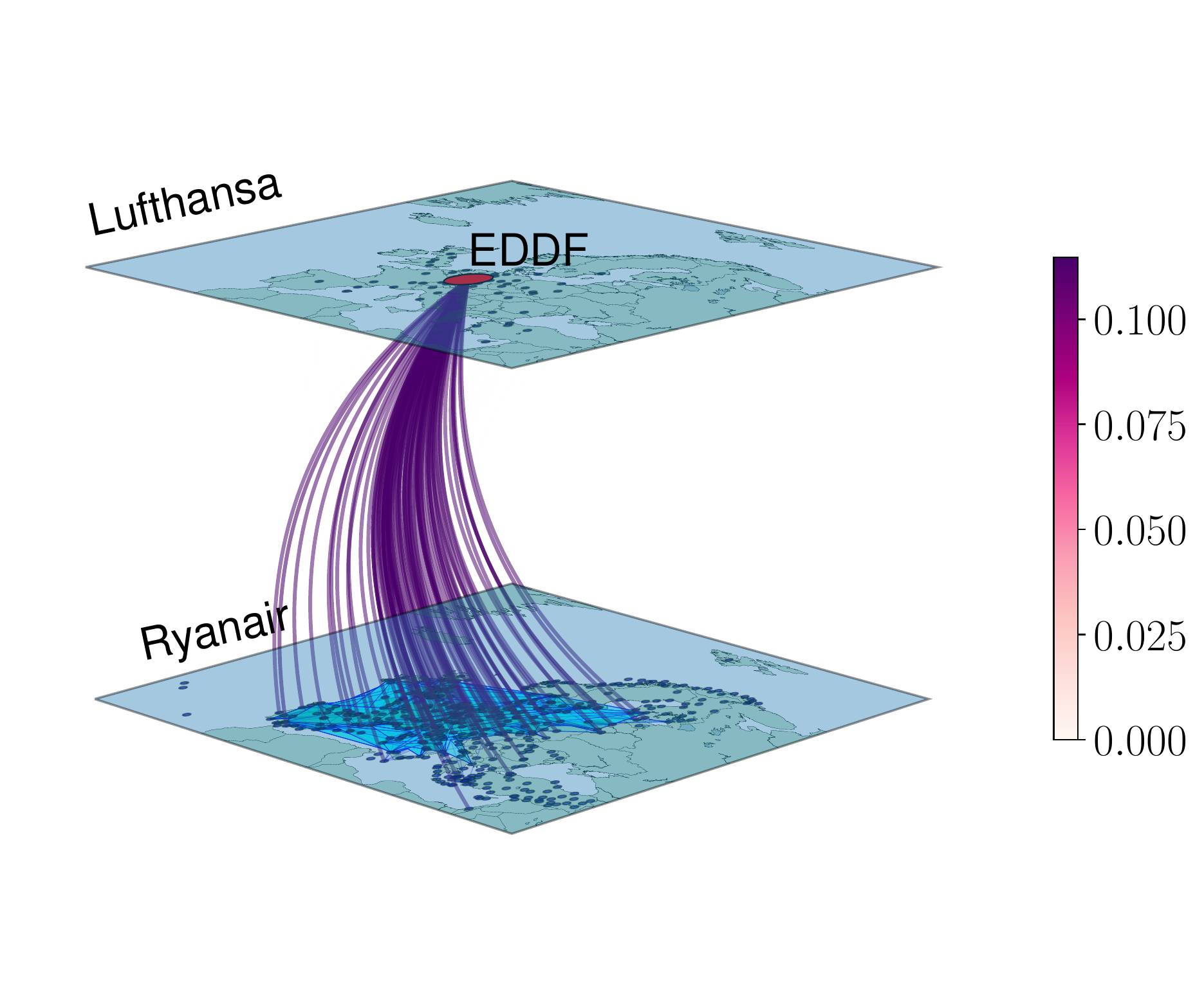} 
 &
  \includegraphics[width=0.42\textwidth]{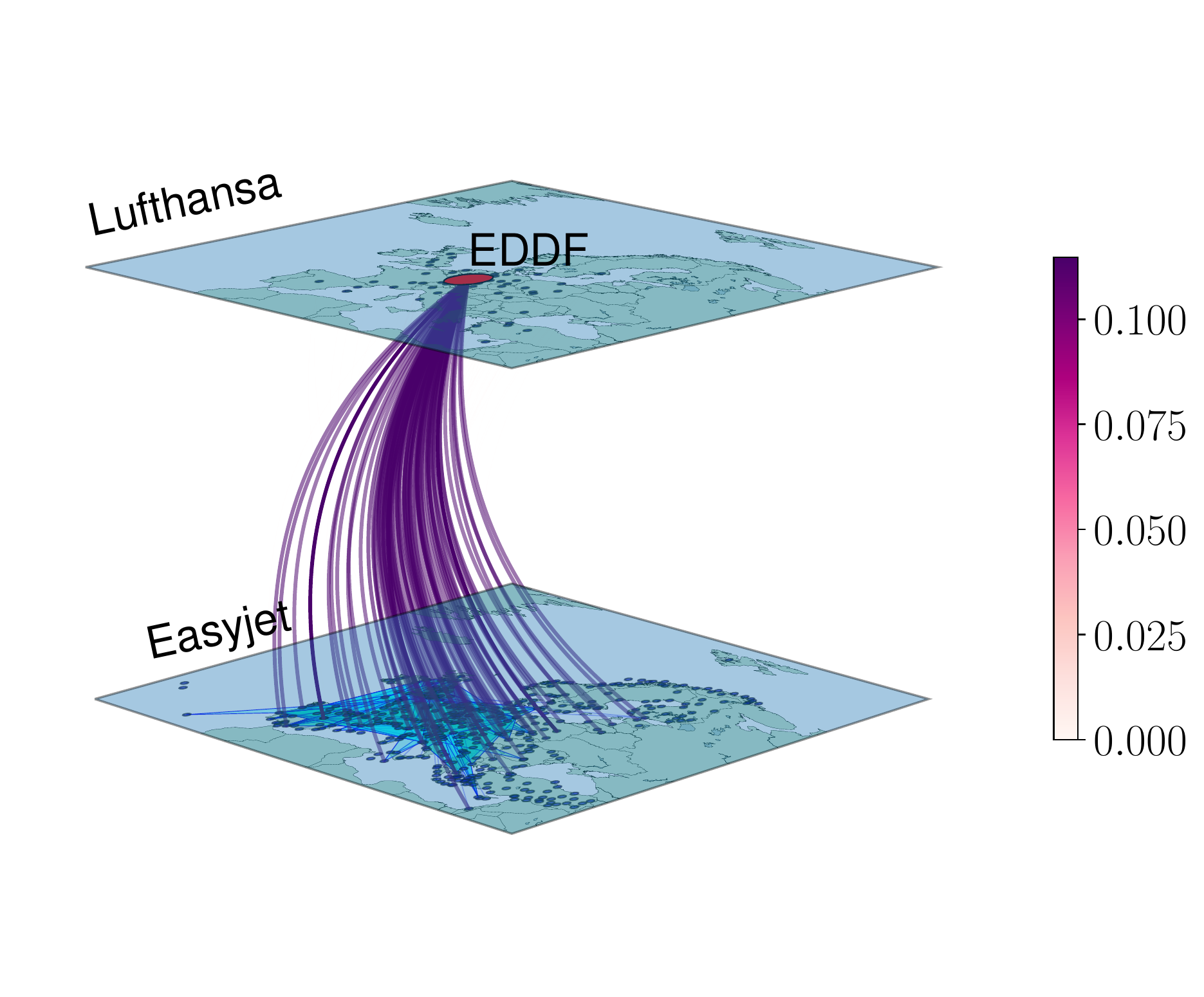}  \\
  Ryanair to Lufthansa & Rynair to Easyjet \\
    \includegraphics[width=0.42\textwidth]{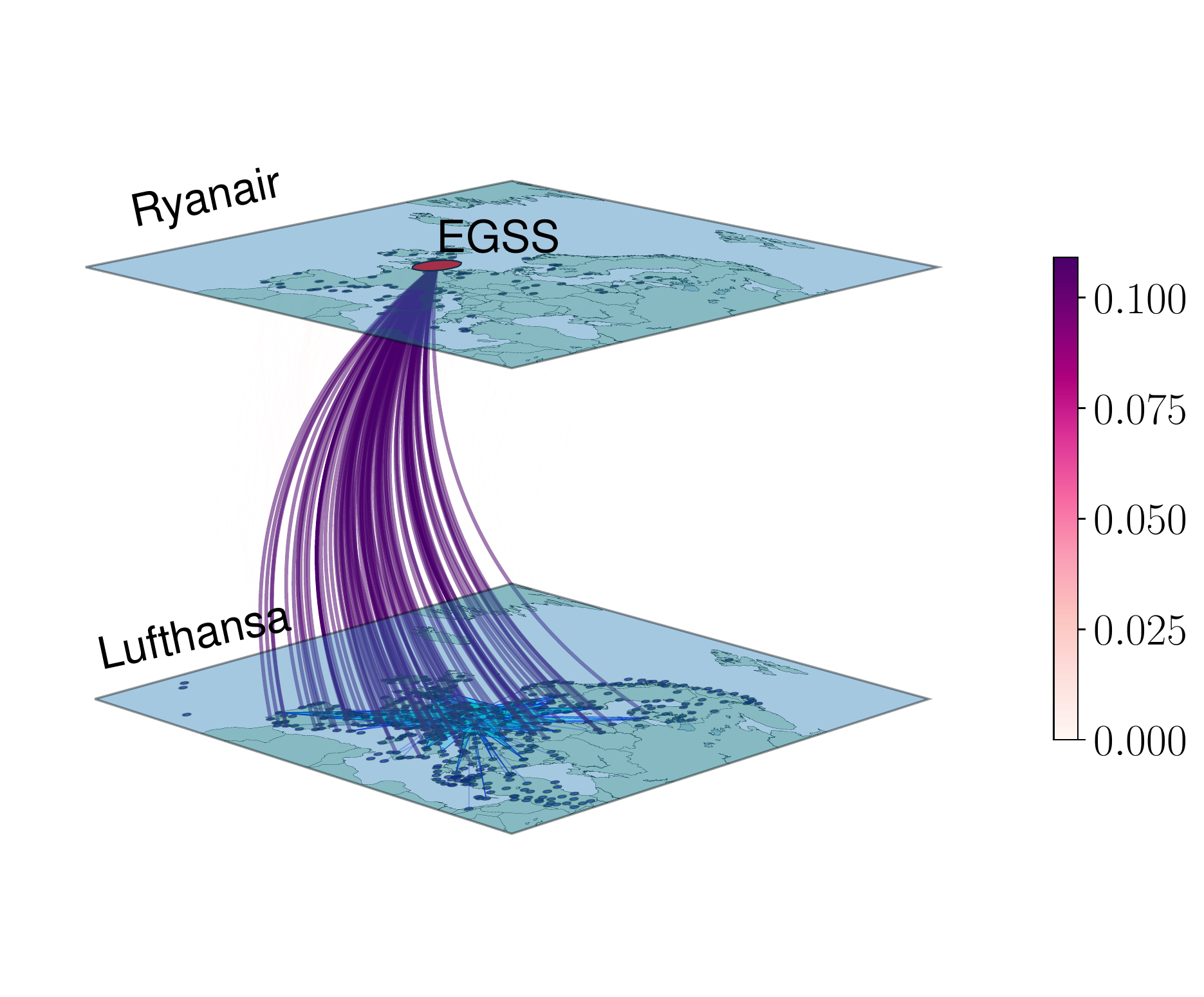}  
    &
    \includegraphics[width=0.42\textwidth]{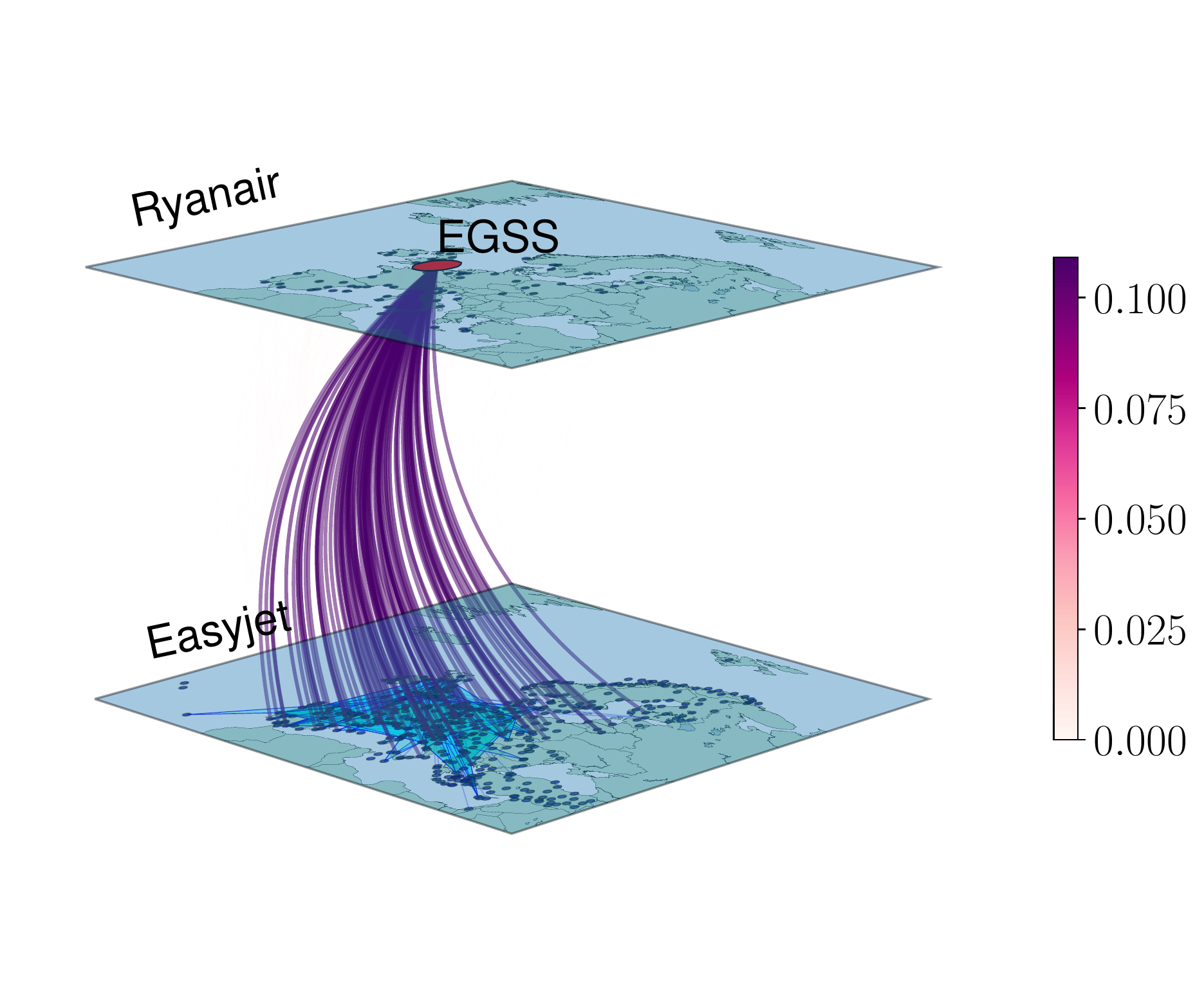} \\
 Easyjet to Lufthansa & Easyjet to Ryanair \\ 
  \includegraphics[width=0.42\textwidth]{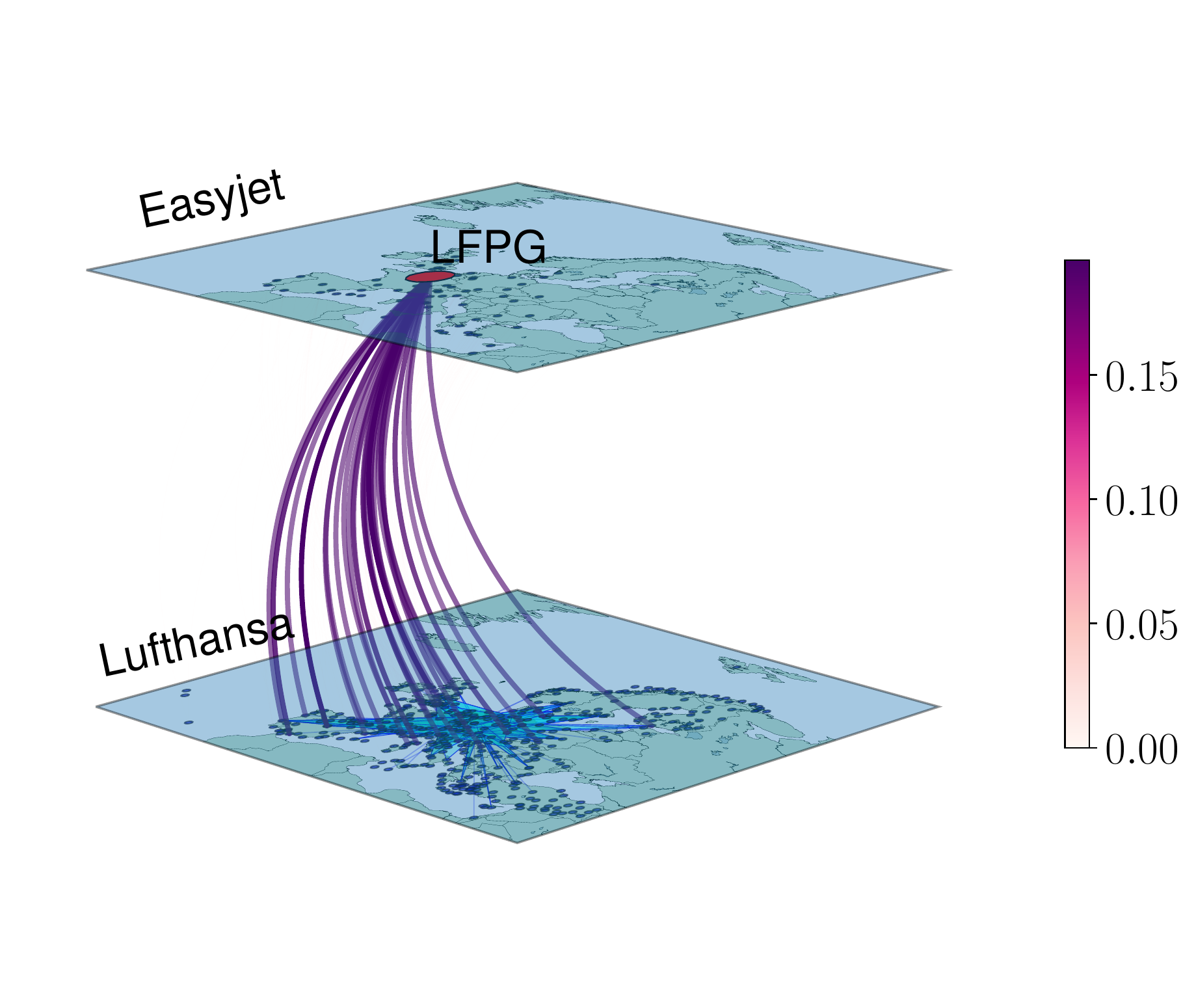}  
  &
    \includegraphics[width=0.42\textwidth]{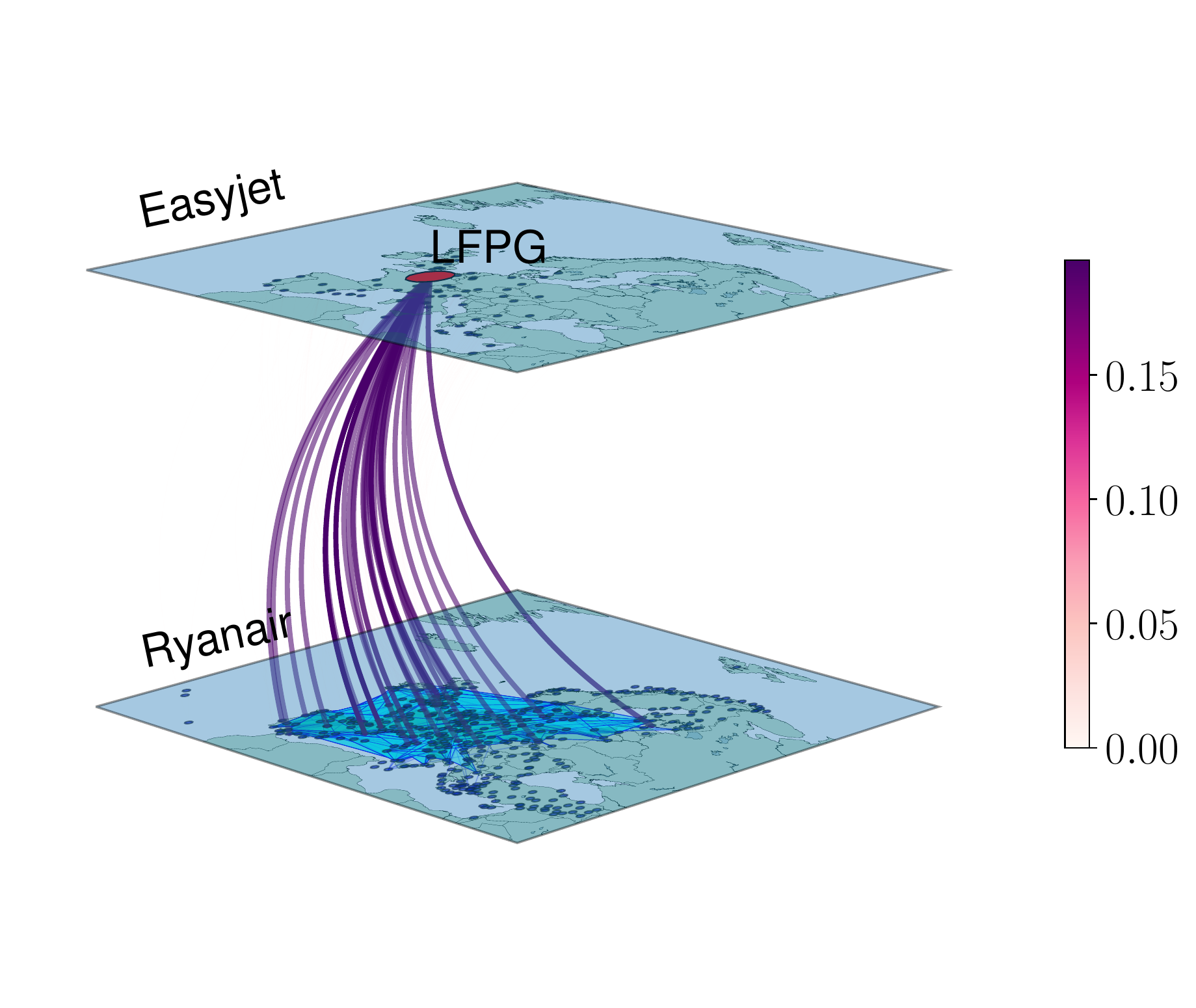}  

\end{tabular}
\caption{{\bf PCHs of the diffusion bicomplexes for the European ATN multiplex.}  The nodes represent airports labelled with their ICAO codes.}
\label{fig:atn-SPB}
\end{figure*}

\begin{table*}[!ht]
\centering
\begin{tabular}{cc}
{\bf 1. $\DiffBicomp{X}{Lufthansa}{Ryanair}$}
 & {\bf  2. $\DiffBicomp{X}{Lufthansa}{Easyjet}$} \\ 
\begin{tabular}{|c|c|}
\hline 
Airports & Rank \\
\hline 
Frankfurt Airport & 1 \\
Munich Airport & 2 \\
D\"usseldorf Airport & 3 \\
Stuttgart Airport & 4 \\
Larnaca Airport & 5 \\
Leipzig Halle Airport & 5 \\
Geneva Airport & 5 \\
Athens Airport & 5 \\
Amsterdam Airport Schiphol & 5 \\
Birmingham Airport & 5 \\
\hline
\end{tabular}
&
\begin{tabular}{|c|c|}
\hline 
Airports & Rank \\
\hline 
Frankfurt Airport & 1 \\
Munich Airport & 2 \\
Milan Malpensa Airport & 3 \\
Düsseldorf Airport & 4 \\
Stuttgart Airport & 5 \\
Larnaca Airport & 6 \\
Leipzig Halle Airport & 6 \\
Geneva Airport & 6 \\
Athens Airport & 6 \\
Amsterdam Airport Schiphol & 6 \\
Birmingham Airport & 6 \\
\hline
\end{tabular}

\\
& \\
{\bf 3. $\DiffBicomp{X}{Ryanair}{Lufthansa}$} & {\bf 4. $\DiffBicomp{X}{Ryanair}{Easyjet}$} \\
\begin{tabular}{|c|c|}
\hline 
Airports & Rank \\
\hline 
London Stansted Airport & 1 \\
Bergamo Airport & 2 \\
Dublin Airport & 3 \\
Charleroi Airport & 4 \\
Paris Beauvais Airport & 5 \\
Porto Airport & 6 \\
Tampere Airport & 7 \\
Kaunas Airport & 8 \\
Zaragoza Airport & 9 \\
G\"oteborg City Airport & 10 \\
\hline
\end{tabular}
&
\begin{tabular}{|c|c|}
\hline 
Airports & Rank \\
\hline 
Bergamo Airport & 1 \\
Dublin Airport & 2 \\
London Stansted Airport & 3 \\
Madrid Barajas Airport & 4 \\
Rome Ciampino Airpor & 5 \\ 
Palma de Mallorca Airport	 & 6 \\
Ibiza Airpor & 7 \\
Bologna Airport	 & 8 \\
Girona Airport & 9 \\
Moss Airport Rygge & 10 \\
\hline
\end{tabular}
\\
& \\
{\bf 5. $\DiffBicomp{X}{Easyjet}{Lufthansa}$} & {\bf 6. $\DiffBicomp{X}{Easyjet}{Ryanair}$} \\
\begin{tabular}{|c|c|}
\hline 
Airports & Rank \\
\hline 
Paris Charles de Gaulle Airport	 & 1 \\
London Stansted Airport & 2 \\
Berlin Brandenburg Airport & 3 \\
London Luton Airport	 & 4 \\
Amsterdam Airport Schiphol & 5 \\
Edinburgh Airport & 6 \\
Manchester Airport  & 7 \\
Rome Fiumicino Airport & 8 \\
Athens Airport & 9 \\
Geneva Airport	& 10\\
\hline
\end{tabular}
&
\begin{tabular}{|c|c|}
\hline 
Airports & Rank \\
\hline 
Paris Charles de Gaulle Airport	 & 1 \\
London Stansted Airport & 2 \\
Milan Malpensa Airport & 3 \\
Berlin Brandenburg Airport & 4 \\
Belfast International Airport & 5 \\
London Luton Airport	 & 6 \\
Amsterdam Airport Schiphol &  7\\
Edinburgh Airport & 8 \\
Newcastle Airport & 9 \\
Lyon-Saint Exupéry Airport & 10 \\
\hline
\end{tabular}

\end{tabular}
\label{tab:EATN-rankings}
\caption{Ranking of the ten most persistent SCHs for the diffusion bicomplexes associated to the European air transportation multiplex network.}

\end{table*}

\section{Discussion and conclusions}
We have introduced CSM as a generalization of both the notions of simplicial complexes and multilayer networks. We further introduced cross-homology  to study their topology and defined the cross--Laplacian operators to detect more structures that are not detected by homology. Our goal here was to set up a mathematical foundation for studying higher-order multilayer complex systems. Nevertheless, through synthetic examples of CSM and applications to multiplex networks, we have shown that our framework provides powerful tools to revealing important topological features in a multilayer networks and address questions that would not arise from the standard pairwise-based formalism of multilayer networks. We put a special focus on the $(0,0)$--cross-Laplacians to show how their spectra quantify the extent to which nodes in one layer restructure the topology of other layers in a multilayer network. Indeed, given a CSB $X$ or even a $2$--layered network, we defined $\Lap{T}{0,0}$ and $\Lap{B}{0,0}$ as two self-adjoint positive operators operators that allow to look at the topology of one layer through the lens of the other layer. Specifically, we saw that their spectra allow to detect nodes from one layer that serve as interlayer connecting hubs for clusters in the other layer; we referred to such nodes as spectral cross-hubs (SCHs). Such hubs vary in function of the eigenvalues of the cross--Laplacians, the notion of {\em spectral persistent cross-hubs} was used to rank them according to their frequency along the spectra. The SCHs obtained from the largest eigenvalues were referred here as {\em principal cross-hubs} (PCHs) as they are the ones that interconnects the most important structures of the other layer. We should note that a PCH is not necessarily spectrally persistent, and two SCHs can be equally persistent but at different ranges of the spectrum. This means that, depending on the applications, some choices need to be made when ranking SCHs based on their spectral persistence. Indeed, it might be the case that two SCHs persist equally longer enough to be considered as the most persistent ones, but that one persists through the first quarter of the spectrum while the other persists through the second quarter of the spectrum, so that none of them is a PCH.
For instance, in the example of the European ATN multiplex networks, when two nodes were equally persistent, we ranked higher the one that came later along the spectrum.
Finally, one can observe that the topological and geometric interpretations given for these operators can be generalized to the higher-order $(k,l)$--cross-Laplacians as well. That is, the spectra of these operators encode the extent to which higher-order topological structures (edges, triangles, tetrahedrons, and so on) control the emergence of higher-order clustering structures in the other layers.

\begin{acknowledgments}
This work was supported by the Natural Sciences and Engineering Research Council of Canada through the CRC grant NC0981.
\end{acknowledgments}

\appendix{

\section{Description of the $(0,0)$--cross-Betti vectors}




\paragraph*{\bf Cones and kites.}

Let $v_j^2$ be a fixed vertex in $V_2$. A {\em kite from $V_1$ to $v_j^2$} is an ordered tuple $(v^1_{i_1}, \ldots, v^1_{i_p})$ of vertices in $V_1$ such that $\{v^1_{i_r}, v^1_{i_{r+1}}, v^2_j\}\in X_{1,0}$ for $r=1, \ldots, p-1$. Such an object is denoted as $(v^1_{i_0}, \ldots, v^1_{i_p} \leftarrow v_j^2)$. Beware that the vertices $v^1_{i_1}, \ldots, v^1_{i_p}$ do not need to be pair-wise connected in $V_1$. What we have here are cross-triangles all pointing to $v^2_j$ that are pieced together in the form of an actual kite as in Figure~\ref{fig:kites-and-cones}. In particular, if $v_j^2$ is the bottom face of a $(1,0)$--cross-triangle $[v_i^1,v_k^1;v_j^2]$, then $(v^1_i,v_k^1\leftarrow v_j^2)$ is a kite. If $(v^1_{i_1}, \ldots, v^1_{i_p}\leftarrow v_j^2)$ is a kite, its {\em boundary} is the triple $(v^1_{i_1}, v_{i_p}^1, v_j^2)\in V_1^2\times V_2$. Similarly, given a fixed vertex $v_i^1\in V_1$, one can define a {\em kite from $V_2$ to $v_i^1$} by a tuple $(v_{j_1}^2, \ldots, v_{j_{p'}}^2)$ of vertices in $V_2$ satisfying analogous conditions. Such a kite will be denoted as $(v_i^1\rightarrow v_{j_1}^2, \ldots, v_{j_{p'}}^2)$.

It is worth noting that if $(v^1_{i_1}, \ldots, v^1_{i_p})$ is a kite from $V_1$ to $v_j^2$, then so is each tuple $(v^1_{i_r}, v^1_{i_{r+1}}, \ldots, v^1_{i_{r+q}})$ with $1\le r$ and $r+q \le p$.


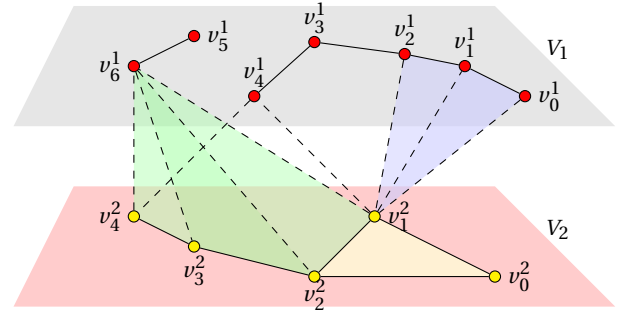
\begin{figure}[!ht]
\centering
\begin{tikzpicture}[scale=.8]

\begin{scope}[on background layer,opacity=0.5]
\fill[gray!20] (0,0) -- (10,0) -- (8,2) node[midway,above, black] {$ \ V_1$} -- (1,2)  -- cycle;
\fill[red!20] (0,-3) -- (10,-3) -- (8,-1)  node[midway,above, black] {$ \ V_2$} -- (1,-1) -- cycle;
\end{scope}

\node[vertex] (u0) at (8.5,0.5) {};
\node at (8.9,0.5) {$v^1_0$};
\node[vertex] (u1) at (7.5,1) {};
\node at (7.5,1.4) {$v^1_1$};
\node[vertex] (u2) at (6.5,1.2) {};
\node at (6.5,1.6) {$v^1_2$};
\node[vertex] (u3) at (5,1.4) {};
\node at (5,1.8) {$v^1_3$};
\node[vertex] (u4) at (4,0.5) {};
\node at (4,0.9) {$v^1_4$};
\node[vertex] (u5) at (3,1.5) {};
\node at (3.4,1.5) {$v^1_5$};
\node[vertex] (u6) at (2, 1) {};
\node at (1.6, 1) {$v^1_6$};

\node[vertex, fill=yellow] (v0) at (8,-2.5) {};
\node at (8.4, -2.5) {$v^2_0$};
\node[vertex, fill=yellow] (v1) at (6,-1.5) {};
\node at (6.4, -1.5) {$v^2_1$};
\node[vertex, fill=yellow] (v2) at (5,-2.5) {};
\node at (5,-2.8) {$v^2_2$};
\node[vertex, fill=yellow] (v3) at (3,-2) {};
\node at (3, -2.4) {$v^2_3$};
\node[vertex, fill=yellow] (v4) at (2, -1.5) {};
\node at (1.6, -1.5) {$v^2_4$};

\draw (u0) -- (u1) -- (u2) -- (u3) -- (u4);
\draw (u5) -- (u6);
\draw (v2) -- (v1) -- (v0) -- (v2) -- (v3) -- (v4);

\draw[dashed] (u0) -- (v1) -- (u1);
\draw[dashed] (u2) -- (v1) -- (u4);
\draw[dashed] (u4) -- (v4) -- (u6) -- (v3);
\draw[dashed] (v1) -- (u6) -- (v2);

\begin{scope}[on background layer]
\fill[blue!20, opacity=0.6] (u0.center) -- (v1.center) -- (u1.center) -- cycle;
\fill[blue!20, opacity=0.5] (u1.center) -- (v1.center) -- (u2.center) -- cycle;
\fill[green!30, opacity=0.5] (u6.center) -- (v3.center) -- (v4.center) -- cycle;
\fill[green!30, opacity=0.5] (u6.center) -- (v1.center) -- (v2.center) -- cycle;
\fill[green!40, opacity=0.5] (u6.center) -- (v2.center) -- (v3.center) -- cycle;
\fill[yellow!20, opacity=0.8] (v0.center) -- (v1.center) -- (v2.center) -- cycle;
\end{scope}

\end{tikzpicture}
\caption{A $2$--dimensional crossimplicial bicomplex containing kites and cones. Indeed $(v^1_0,v^1_1,v^1_2)$ is kite from $V1$ to $v^2_1\in V_2$ with boundary $(v^1_0,v^1_2,v^2_1)\in V_1^2\times V_2$ and $(v^2_1, v^2_2,v^2_3,v^2_4)$ is a kite from $V_2$ to $v^1_6\in V_1$ with boundary $(v^1_6,v^2_1,v^2_4)\in V_1\times\in V_2^2$. The tuples $(v^2_1,v^2_2, v^2_3), (v^2_1,v^2_2)$, and $(v^2_2,v^2_3,v^2_4)$ are also kites from $V_2$ to $v^1_6$. Furthermore, there are $3$ cones with bases in $V_1$: $(v^1_2,v^1_4,v^2_1)$ is a closed cone with base $(v^1_2,v^1_4)\in V_1$ and vertex $v^2_1\in V_2$, and $(v^1_4,v^1_6,v^2_1)$ is an open cone with base $(v^1_4,v^1_6)\in V_1^2$ and vertex $v^2_1\in V_2$. Also, $(v^1_4,v^1_6,v^2_4)$ is an open cone with base $(v^1_4,v^1_6)\in V_1^2$ and vertex $v^2_4\in V_2$; and $(v^2_1,v^2_4,v^1_4)\in V_2^2\times V_1$ is a closed cone with base $(v^2_1,v^2_4)\in V_2^2$ and vertex $v^1_4\in V_1$. It follows from Theorem~\ref{thm:cross-Betti} that $\bettiVect{0,0}=(3,1)$.}
\label{fig:kites-and-cones}
\end{figure}

By a {\em cross-chain} on a kite we mean one that is a linear combination of the triangles composing the kite; that is, a cross-chain on the kite $(v^1_{i_1}, \ldots, v^1_{i_p}\leftarrow v_j^2)$ is an element $a\in C_{1,0}(X)$ of the form 
\begin{eqnarray}\label{eq:kite-cross-chain}
a = \sum_{r=1}^{p-1}\g_r[v_{i_r}^1,v_{i_{r+1}}^1; v_j^2],
\end{eqnarray}
where $\g_1, \ldots, \g_{p-1}\in \bb R$. In a similar fashion, cross-chains on a kite of the form $(v_i^1\rightarrow v_{j_1}^2, \ldots, v_{j_{p'}}^2)$ are defined.

Now, given a pair $(v^1_i, v^1_k)$ of vertices in the layer $V_1$ and the vertex $v_j^2\in V_2$, we say that the triple $(v_i^1,v^1_k, v_j^2)\in V_1^2\times V_2$ is a {\em cone with base $(v_i^1,v_k^1)$ and vertex $v_j^2$} if it satisfies the following conditions:
\begin{itemize}
  \item $v_i^1\sim v_j^2$ and $v_k^1\sim v_j^2$; {\em i.e.}, $[v_i^1;v_j^2], [v_k^1;v_j^2] \in X_{0,0}$;
  \item the triple $(v_i^1,v_k^1,v_j^2)\in V_1^2\times V_2$ is not the boundary of a kite from $V_1$ to $v_j^2$.
\end{itemize}
We also say that $(v_i^1,v_k^1,v_j^2)$ is a {\em cone with base in $V_1$ and vertex in $V_2$}. In a similar fashion one defines a cone with base in $V_2$ and vertex in $V_1$. We refer to Figure~\ref{fig:kites-and-cones} for examples of cones.

An immediate consequence of a triple $(v_i^1,v_k^1,v_j^2)\in V_1^2\times V_2$ being a cone is that the vertices $\{v_i^1, v_k^1, v_j^2\}$ is not a $(1,0)$--crossimplex. The vertices $v_i^1$ and $v_k^1$  might however be connected by a {\em horizontal path} of some length; by which we mean that there might be a sequence of vertices $v_{i_0}^1, \ldots, v_{i_p}^1$ in $V_1$ not all of which form cross-triangles with $v_j^2$ and such that 
\begin{eqnarray*}
v_i^1\outAdj{1}{} v_{i_0}^1 \outAdj{1}{} \cdots \outAdj{1}{} v^1_{i_p} \outAdj{1}{} v_k^1,
\end{eqnarray*}
in which case the cone is said to be {\em closed}; it is called {\em open} otherwise. 

Cones in a crossimplicial bicomplex are classified by the top and bottom $(0,0)$--cross-homology groups of the bicomplex. Specifically, we have the following topological interpretation of $\hH{0,0}(X)$, $\vH{0,0}(X)$, and hence, the $(0,0)$-cross-Betti numbers.

\begin{thm}\label{thm:cross-Betti}
The $(0,0)$--cross-homology group $\hH{0,0}(X)$ (resp. $\vH{0,0}(X)$) is generated by the cross-homology classes of cones with bases in $V_1$ and vertices in $V_2$ (resp. with bases in $V_2$ and vertices in $V_1$). Therefore, the $(0,0)$--cross-Betti number $\betti{t}{0,0}$ counts the generating cones with bases in $V_t, t=1,2$.
\end{thm}
Here, by the cross-homology class of the cone $(v_i^1,v_k^1,v_j^2)\in V_1^2\times V_2$, for instance, we mean the top cross-homology of the $(0,0)$--cross-chain $[v_k^1;v_j^2]-[v_i^1;v_j^2]\in C_{0,0}(X)$. 

\begin{proof}
We prove the theorem for $\hH{0,0}(X)$ since the same arguments apply to $\vH{0,0}(X)$. Every cone $(v_i^1,v_k^1,v_j^2)$ defines a non-trivial $(0,0)$--cross-cycle; namely, the difference of the corresponding cross-edges $[v_i^1;v_j^2]-[v_k^1;v_j^2]\in \ker \hbd{0,0}$. More generally, suppose we are given $p$ cones $(v_{i_1}^1,v_{i_2}^1,v_j^2), (v_{i_2},v_{i_3}^1,v_j^2), \ldots, (v_{i_{p-1}}^1,v_{i_p}^1,v_j^2)$ with bases in $V_1$ and all with the same vertex $v_j^2\in V_2$. Then, for all real numbers $\a_1, \ldots, \a_p$ such that $\sum_{r=1}^p\a_r=0$, the cross-chain 
\begin{eqnarray}\label{eq:cone-cross-cycle}
  b=\sum_{r=1}^p\a_r[v_{i_r}^1;v_j^2]
\end{eqnarray}
is clearly a $(0,0)$--cross-cycle with non-trivial cross-homology class; {\em i.e.}, $b\in \ker\hbd{0,0}$ and $b\notin \im\hbd{1,0}$. Conversely, let $b'\in \ker\hbd{0,0}$. We can write
\[
b'=\sum_{m=1}^M\a_m'[v_{i_m}^1;v_{i_m}^2]\in C_{0,0}(X),
\] 
so that $\hbd{0,0}(b')=\sum_{m=1}^M\a'_m[v_{i_m}^2]=0$. Then, either all the $v_{i_m}^2$'s are pair-wise different, in which case $b'$ is the trivial cross-cycle; or there exist $p+1$ subsets $(\{m_{r,1}, \ldots, m_{r, M_r}\})_{r=1}^{p+1}$ of $\{1, \ldots, M\}$ such that 
\[
v_{i_{m_{r,1}}}^2 = v_{i_{m_{r,2}}}^2 = \cdots = v^2_{i_{m_{r, M_r}}}, \ {\rm for\ } 1\le r \le p, 
\]
and 
\[
v_{i_{m_{p+1,j}}}^2\neq v_{i_{m_{p+1,j'}}}^2, \ {\rm for\ all\ } j\neq j', 1\le j,j' \le M_{p+1}.
\]
It follows that 
\begin{eqnarray}\label{eq:cross-cycles-coefficients}
  \sum_{j=1}^{M_r}\a'_{m_{r,j}}=0, \ {\rm for\ each\ } r=1, \ldots, p,
\end{eqnarray}
and $\a'_{m_{p+1,j}} = 0 \ {\rm for\ all\ } j=1, \ldots, M_{p+1}$. Hence, we get the following general expression of a $(0,0)$--cross-cycle:
\begin{eqnarray}\label{eq:cross-cycles-expression}
b'=\sum_{r=1}^p\sum_{j=1}^{M_r}\a'_{m_{r,j}}[v_{i_{m_{r,j}}}^1;v_{i_{m_{r,1}}}^2],
\end{eqnarray}
where the coefficients satisfy~\eqref{eq:cross-cycles-coefficients}. Furthermore, it is straightforward to see that $b'\in \im\hbd{1,0}$ if and only if for each $r=1, \ldots, p$, there exists a permutation $\tau_r$ of $\{1, \ldots, M_r\}$ such that 
\[
(v_{i_{m_{r, \tau_r(1)}}}^1, \ldots, v_{i_{m_{r,\tau_r(M_r)}}}^1 \leftarrow v_{i_{m_{r,1}}}^2)
\]
is a kite. In that case, we get $b'=\hbd{1,0}(a)$ where 
\begin{eqnarray}\label{eq:trivial-cross-cycle}
a=\sum_{r=1}^p\sum_{r=1}^{M_r-1} \g_{m_{r,j}}[v_{i_{m_{r,\tau_r(j)}}}^1, v_{i_{m_{r,\tau_r(j+1)}}}^1;v_{i_{m_{r,1}}}^2],
\end{eqnarray}
and where for $r=1, \ldots, p$, the coefficients $\g_{m_{r,j}}$ are given by
\begin{eqnarray}
  \left\{
  \begin{array}{ll}
    \g_{m_{r,1}} & = -\a'_{m_{r,\tau_r(1)}} \\
    \g_{m_{r,2}} & = -\a'_{m_{r,\tau_r(1)}} -\a'_{m_{r,\tau_r(2)}}
    \\
     & \vdots \\
     \g_{m_{r,M_r-2}} & = -\a'_{m_{r,\tau_r(1)}} -\a'_{m_{r,\tau_r(2)}} - \cdots - \a'_{m_{r,\tau_r(M_r-2)}}\\
     \g_{m_{r,M_r-1}} & = \a'_{m_{r,\tau_r(M_2)}}
  \end{array}
  \right.
\end{eqnarray}
This shows that trivial cross-homology classes in $\hH{0,0}(X)$ are given by cross-cycles obtained from cross-chains on kites; that is images of sums of cross-chains in the form of~\eqref{eq:kite-cross-chain}.
\end{proof}

\label{ap:cones}

\section{Hodge Theory}
\subsection{Harmonic cross-forms as cross-homology classes}\label{ap:harmonic}
Let
  \[
    \Lap{TO}{k,l} := (\hcobd{k,l})^*\hcobd{k,l}, \ \Lap{TI}{k,l} := \hcobd{k-1,l}(\hcobd{k-1,l})^*,
  \]
and
  \[
   \Lap{BO}{k,l} := (\vcobd{k,l})^*\vcobd{k,l}, \ \Lap{BI}{k,l} := \vcobd{k,l-1}(\vcobd{k,l-1})^*,
  \]
so that 
  \[
   \Lap{T}{k,l} := \Lap{TO}{k,l} + \Lap{TI}{k,l};
  \]
and 
  \[
   \Lap{B}{k,l} := \Lap{BO}{k,l} + \Lap{BI}{k,l}.
  \]

We denote the spaces of harmonic cross--forms on $X$ as
\[
 \har{s}{k,l}=\ker\Lap{s}{k,l}= \{\phi \in C^{k,l} | \Lap{s}{k,l}\phi = 0\}, s=T,B
\]

Then, we have the following group isomorphisms generalizing~\cite{eckmann1944, horak2013spectra}. 
\begin{lem}
For $s=1,2$ and for all  $k,l\ge -1$, we have
\begin{eqnarray}\label{eq:isom-cohom-harmonic}
\xH{s}{k,l}(X) \cong \ker (\Lap{s}{k,l}),
\end{eqnarray}
where we have used the notations $\tLap{k,l} = \Lap{T}{k,l}$ and $\bLap{k,l}=\Lap{B}{k,l}$. 
\end{lem}
\begin{proof}
Let's prove this result for $s=1$ (similar arguments apply to $s=2$). First, notice that from the identification $C_{k,l}=C^{k,l}$ we obtain
\begin{eqnarray}\label{eq:cohom-decomposition}
\hH{k,l}(X)=\ker(\hcobd{k,l})/\im(\hcobd{k-1,l}) \cong \ker(\hcobd{k,l})\cap \im(\hcobd{k-1,l})^\bot,
\end{eqnarray}
and as the analog holds for $\vcoh{k,l}(X)$. Moreover, recall from linear algebra that if $E\overset{f}{\rTo}F$ is a linear operator on two vector spaces equipped with inner products, then $\ker(f^*f)=\ker f$. Indeed, we clearly have $\ker f\subset \ker(f^*f)$. Next, if $x\in \ker(f^*f)$, then $\inner{f^*fx,y}{E}=\inner{fx,fy}{F}=0$ for all $y\in E$, which implies that $x\in \ker f$. In our case we have $\hcobd{k,l}\hcobd{k-1,l}=0$ and $(\hcobd{k-1,l})^*(\hcobd{k,l})^*=0$; hence
\[
\begin{aligned}
\im(\Lap{TO}{k,l}) \subset \im(\hcobd{k,l})^* \subset \ker(\hcobd{k-1,l})^*\subset \ker(\hcobd{k-1,l}(\hcobd{k-1,l})^*) \\
\im(\Lap{TI}{k,l}) \subset \im(\hcobd{k-1,l}) \subset \ker(\hcobd{k,l}) \subset \ker((\hcobd{k,l})^*\hcobd{k,l}).
\end{aligned}
\]
Therefore
\[
\begin{array}{lcl}
  \ker\Lap{T}{k,l} & = & \ker((\hcobd{k,l})^*\hcobd{k,l}) \cap \ker(\hcobd{k-1,l}(\hcobd{k-1,l})^*) \\
  & = & \ker(\hcobd{k,l})\cap \ker(\hcobd{k-1,l})^* \\
  & \cong & \ker(\hcobd{k,l})\cap \im(\hcobd{k-1,l})^\bot,
\end{array}
\]
and the isomorphism~\eqref{eq:isom-cohom-harmonic} follows from~\eqref{eq:cohom-decomposition}. 
\end{proof}

It follows that the eigenvectors corresponding to the zero eigenvalue of the $(k,l)$--cross-Laplacian $\Lap{s}{k,l}$ are representative cross-cycles in the homology group $\xH{s}{k,l}(X)$. Henceforth, we see that in order to get the dimensions of the cross-homology groups $\xH{s}{k,l}(X)$, it suffices to find the eigenspaces corresponding to the zero eigenvalues of $\Lap{s}{k,l}$. That is, 

\begin{eqnarray}\label{eq:betti-and-harmonic}
  \betti{1}{k,l} = \dim \ker\Lap{1}{k,l}, \ {\rm and\ } \betti{2}{k,l} = \dim \ker\Lap{2}{k,l}.
\end{eqnarray}

\label{ap:harmonic}
\subsection{Matrix representations of the cross--Laplacians}\label{ap:matrices}
In order to compute these matrix representations of the $(k,l)$--cross-Laplacians, we first need to give their formal expressions as linear operators. Thanks to~\eqref{eq:dual-coboundary-formula}, we get for $\phi\in C^{k,l}$ and $a\in X_{k,l}$:

\begin{widetext}
\begin{align*}
(\hcobd{k,l})^*\hcobd{k,l}\phi([a]) & = \sum_{\overset{a'\in X_{k+1,l}}{a\in \hbd{}a'}}\frac{w(a')}{w(a)}\sgn(a,\hbd{}a)(\hcobd{k,l}\phi)([a']) \\
 & = \sum_{\overset{a'\in X_{k+1,l}}{a\in \hbd{}a'}}\frac{w(a')}{w(a)}\sgn(a,\hbd{}a')\sum_{b\in \hbd{}a'}\sgn(b,\hbd{}a')\phi([b]) \\
 & =  \sum_{\overset{a'\in X_{k+1,l}}{a\in \hbd{}a'}}\frac{w(a')}{w(a)}\sgn(a,\hcobd{}a')\left[\sgn(a,\hbd{}a')\phi([a]) +\sum_{b\in \hbd{}a', b\neq a}\sgn(b,\hbd{}a')\phi([b]) \right] \\
 & = \sum_{\overset{a'\in X_{k+1,l}}{a\in \hbd{}a'}}\frac{w(a')}{w(a)}\left[\phi([a]) + \sum_{b\in \hbd{}a', a\neq b}\sgn(a,\hbd{}a')\sgn(b,\hbd{}a')\phi([b])\right],
\end{align*}
and
\begin{align*}
  \hcobd{k-1,l}(\hcobd{k-1,l})^*\phi([a]) & = \sum_{\overset{c\in X_{k-1,l}}{c\in \hbd{}a}}\sgn(c,\hbd{}a')\sum_{\overset{a'\in X_{k,l}}{c\in \hbd{}a'}}\frac{w(a')}{w(c)}\sgn(c,\hbd{}a')\phi([a']) \\
  & = \sum_{\overset{c\in X_{k-1,l}}{c\in \hbd{}a}}\sgn(c,\hbd{}a)\left[\frac{w(a)}{w(c)}\sgn(c,\hbd{}a)\phi([a]) \right.  \\
  & + \left. \sum_{\overset{c\in \hbd{}a'}{a\neq a'}}\frac{w(a')}{w(c)}\sgn(c,\hbd{}a')\phi([a'])\right] \\
  & = \sum_{\overset{c\in X_{k-1,l}}{c\in \hbd{}a}}\frac{w(a)}{w(c)}\phi([a]) + \sum_{\overset{c\in X_{k-1,l}, a'\in X_{k,l}}{c=\hbd{}a'\cap\hbd{}a}}\frac{w(a')}{w(c)}\sgn(c,\hbd{}a)\sgn(c,\hbd{}a')\phi([a']).
\end{align*}
\end{widetext}
In particular, when $\phi$ is an elementary cross-form $e_b, b\in X_{k,l}$, we get

\begin{widetext}
\begin{eqnarray*}
\Lap{TO}{k,l}e_b([a]) = \left\{
\begin{array}{ll}
\frac{1}{w(a)}\deg_{TO}(a), & {\rm if\ } a=b, \\
 & \\
-\frac{w(c)}{w(a)}, & {\rm if \ } a\neq b \ {\rm and \ } a\outAdj{1}{c}b, \\ 
& {\rm and\ } \sgn(a)=\sgn(b), \\  
& \\ 
\frac{w(c)}{w(a)}, & {\rm if \ } a\neq b \ {\rm and \ } a\outAdj{1}{c}b, \\ 
& {\rm and\ } \sgn(a)=-\sgn(b), \\
 & \\
0, & {\rm otherwise\ }, 
\end{array}\right.
 \ {\rm and \ }
  \Lap{TI}{k,l}e_b([a]) = \left\{
  \begin{array}{ll} 
  w(a)\deg_{TI}(a), & {\rm if\ } a=b, \\
   & \\
  \frac{w(b)}{w(d)}, & {\rm if\ } a\neq b \ {\rm and\ } a\inAdj{1}{d}b,\\
   & {\rm and\ } \sgn(a)=\sgn(b),\\ 
   & \\ 
   - \frac{w(b)}{w(d)}, & {\rm if\ } a\neq b \ {\rm and\ } a\inAdj{1}{d}b \\ 
   & {\rm and\ } \sgn(a) = -\sgn(b),\\ 
 & \\
0, & {\rm otherwise\ }.
   \end{array}\right.
\end{eqnarray*}
\end{widetext}

It follows that the $(a,b)$-th entry of the matrix representation of the top $(k,l)$--cross-Laplacian $\Lap{T}{k,l}$ with respect to the inner product defined from the weight $w$ on $X$ is given by
\begin{widetext}
\begin{eqnarray}\label{eq:matrix-representation}
  (\Lap{T}{k,l})_{a,b} = \left\{
  \begin{array}{l}
    \frac{1}{w(a)}\deg_{TO}(a)+w(a)\deg_{TI}(a), \ {\rm if\ } a=b, \\
     \\
     \begin{array}{ll}
    \frac{w(b)}{w(d)} - \frac{w(c)}{w(a)} , &  {\rm if \ } a\neq b, \ a\outAdj{1}{c}b, \ a\inAdj{1}{d}b, \\
         & {\rm and\ } \sgn(a)=\sgn(b),\\
      \end{array}
        \\ 
       \begin{array}{ll}
     \frac{w(c)}{w(a)} - \frac{w(b)}{w(d)}, &  {\rm if \ } a\neq b, \  a\outAdj{1}{c}b,\ a\inAdj{1}{d}b, \\
     & {\rm and\ } \sgn(a)=-\sgn(b),  \\
       \end{array}
        \\
       \begin{array}{ll}
       \frac{w(b)}{w(d)}, &  {\rm if\ } a\neq b, \ a\inAdj{1}{d}b, \ \sgn(a)=\sgn(b), \\  
       & {\rm and\ not \ top-outer \ adjacent}, \\ 
       \end{array}
        \\
        \begin{array}{ll}
       -\frac{w(b)}{w(d)}, & {\rm if\ } a\neq b, \ a\inAdj{1}{d}b, \ \sgn(a)=-\sgn(b), \\  & {\rm and\ not \ top-outer \ adjacent}, \\
       \end{array} 
        \\
       0, \ {\rm otherwise\ }.  
  \end{array}
  \right.
\end{eqnarray}
\end{widetext}
And it is clear that we get similar matrix representation for the bottom $(k,l)$--cross-Laplacian $\Lap{B}{k,l}$. 

\label{ap:matrices}

}

\bibliographystyle{abbrv} 

\end{document}